\documentclass[sn-mathphys,Numbered]{sn-jnl}

\usepackage{multirow,diagbox,multicol,booktabs}
\usepackage{graphicx}%
\usepackage{amsmath,amssymb,amsfonts}%
\usepackage{amsthm}%
\usepackage{mathrsfs}%
\usepackage{xcolor}%
\usepackage{textcomp}%
\usepackage{manyfoot}%
\usepackage{algorithm,algorithmic}%

\usepackage{latexsym,enumerate,verbatim,amsfonts,microtype}
\usepackage{color} 
\usepackage{graphicx}
\usepackage[active]{srcltx}
\usepackage{cases}
\usepackage[colorinlistoftodos,prependcaption,textsize=footnotesize]{todonotes}
\usepackage[framemethod=tikz]{mdframed}
\mdfsetup{%
	skipbelow=4pt,
	skipabove=8pt,
	linewidth=1.25pt,
	backgroundcolor=gray!10,
	userdefinedwidth=\textwidth,
	roundcorner=10pt,
}

\allowdisplaybreaks[2]
\numberwithin{equation}{section}

\def\R{{\mathbb{R}}}
\def\A{{\mathfrak{A}}}

\def\Argmin{\mathop{\rm Arg\,min}}

\def\conv{{\rm conv}\,}

\def\d{{\rm dist}}
\def\dom{{\rm dom}\,}
\def\xorig{{x_{\rm orig}}}

\newcommand{\ling}[3]{{\rm lin}_{g_{#1}}(#2,#3)}

\newtheorem{theorem}{Theorem}[section]
\newtheorem{definition}{Definition}[section]%
\newtheorem{lemma}{Lemma}[section]

\newtheorem{proposition}{Proposition}[section]
\newtheorem{assumption}{Assumption}[section]

\newtheorem{example}{Example}[section]%
\newtheorem{remark}{Remark}[section]%

\begin{document}

\title[ESQM]{An extended sequential quadratic method with extrapolation}


\author[1]{\fnm{Yongle} \sur{Zhang}}

\author*[2]{\fnm{Ting Kei} \sur{Pong}}\email{tk.pong@polyu.edu.hk}

\author[3]{\fnm{Shiqi} \sur{Xu}}

\affil[1]{\orgdiv{Department of Mathematics}, \orgname{Visual Computing and Virtual Reality Key Laboratory of Sichuan Province, Sichuan Normal University}, \orgaddress{\city{Chengdu}, \country{People's Republic of China}}}

\affil*[2]{\orgdiv{Department of Applied Mathematics}, \orgname{the Hong Kong Polytechnic University}, \orgaddress{\city{Hong Kong}, \country{People's Republic of China}}}

\affil[3]{\orgdiv{Department of Mathematics}, \orgname{Sichuan Normal University}, \orgaddress{\city{Chengdu}, \country{People's Republic of China}}}


\abstract{We revisit and adapt the extended sequential quadratic method (ESQM) in \cite{Ausleder13} for solving a class of difference-of-convex optimization problems whose constraints are defined as the intersection of level sets of Lipschitz differentiable functions and a simple compact convex set. Particularly, for this class of problems, we develop a variant of ESQM, called ESQM with extrapolation (ESQM$_{\rm e}$), which incorporates Nesterov's extrapolation techniques for empirical acceleration. Under standard constraint qualifications, we show that the sequence generated by ESQM$_{\rm e}$ clusters at a critical point if the extrapolation parameters are uniformly bounded above by a certain threshold. Convergence of the whole sequence and the convergence rate are established by assuming Kurdyka-{\L}ojasiewicz (KL) property of a suitable potential function and imposing additional differentiability assumptions on the objective and constraint functions. In addition, when the objective and constraint functions are all convex, we show that linear convergence can be established if a certain exact penalty function is known to be a KL function with exponent $\frac12$; we also discuss how the KL exponent of such an exact penalty function can be deduced from that of the \emph{original} extended objective (i.e., sum of the objective and the indicator function of the constraint set). Finally, we perform numerical experiments to demonstrate the empirical acceleration of ESQM$_{\rm e}$ over a basic version of ESQM, and illustrate its effectiveness by comparing with the natural competing algorithm SCP$_{\rm ls}$ from \cite{yu21}.}




\maketitle

\section{Introduction}

Extrapolation techniques, due to their simplicity and easy adaptability, have been widely studied in recent years to empirically accelerate first-order methods; see, for example, \cite{Polyak64,SOR00,nesterov83,Brezinski00,SmithFordSidi87} and references therein. Among them, Nesterov's extrapolation techniques \cite{nesterov83,Nesterov2004,Nesterov2005,Nesterov2013} have been successfully applied to accelerate the proximal gradient algorithm \cite{LionsMercier79} for minimizing $f + h$, with $f$ being a convex {\em loss function} with Lipschitz continuous gradient, and $h$ being a proper closed convex and possibly nonsmooth {\em regularizer} with \emph{easy-to-compute} proximal operator. These studies led to the developments of various algorithms and softwares including the well-known algorithm FISTA \cite{beck09} for linear inverse problems and the software TFOCS \cite{BeckerCandesGrant11} for solving a large class of convex cone problems. Nesterov's extrapolation techniques have also been suitably adapted in subsequent works such as \cite{wen17,wen18} in some nonconvex settings, and most of these works also require the proximal operator of (part of) the regularizer to be easy to compute: in the case when $h$ is the indicator function of some closed set $D$, this requirement amounts to saying that the projection onto $D$ can be computed efficiently. In this paper, we consider a class of constrained optimization problems whose constraint sets \emph{do not admit easy projections}, and investigate the adaptation of extrapolation techniques on empirically accelerating a classical algorithm for these problems.

Specifically, we consider the following difference-of-convex (DC) optimization problem with smooth inequality and simple geometric constraints:
\begin{align}\label{eq1}
\min_{x\in \mathbb {R}^n }\quad& P(x):= P_1(x) - P_2(x) \notag \\
\text{s.t.}\quad & g_i(x)\leq 0,~~ i = 1,\ldots, m,\\
& x\in C,  \notag
\end{align}
where $P_1:\mathbb{R}^n \to \R$ and $P_2:\mathbb{R}^n\to \R$ are convex, each $g_i:\mathbb{R}^n\to \R$ is smooth and $\nabla g_i$ is Lipschitz continuous, $C\subseteq \mathbb {R}^n$ is a nonempty compact convex set, and the feasible set $C\cap \mathscr{F}$ is nonempty, where $\mathscr{F}:=\{x\in \mathbb{R}^n : g_i(x)\leq 0, ~i = 1,\dots,m\}$. This class of problems arises naturally in many applications. For example, in compressed sensing, the $P$ can be a sparsity inducing regularizer such as the difference of $\ell_1$ and $\ell_2$ norms \cite{YinLouHeXin15}, the $g_i$, $i = 1,\ldots,m$, can be loss functions based on the noise models in the $i$th transmission channel, and the set $C$ can be used to model some priors such as nonnegativity or boundedness.

Since projections onto the feasible set of \eqref{eq1} are not easy to compute, existing algorithms for \eqref{eq1} usually leverage the Lipschitz continuity of $\nabla g_i$ to build approximations for the feasible sets, leading to relatively easier subproblems. One natural approach for building approximations is to replace $g_i$ by its quadratic majorants at the current iterate. Specific algorithms based on quadratically approximating $g_i$ in \eqref{eq1} include the moving balls approximation algorithm \cite{AusSheTeb10} and its variants (see, e.g., \cite{bolte16,yu21}). Another natural approach for building approximations is to make use of \emph{affine approximations} to $g_i$ at the current iterate, leading to subproblems with even simpler structures. This approach has its roots in
the literature of sequential quadratic programming (SQP) method, and we refer the readers to \cite{GillWong12} and references therein for more discussions of SQP. Here, we are interested in the framework described in \cite{Ausleder13}, which focused on solving \eqref{eq1} when $P$ and each $g_i$ are twice continuously differentiable. We adapt the algorithmic framework described there to solve \eqref{eq1}, and incorporate extrapolation techniques to empirically accelerate the algorithm. We call the resulting algorithm extended sequential quadratic method with extrapolation (ESQM$_{\rm e}$), following the use of the name ESQM in \cite{Ausleder13}. The algorithmic details will be presented in Section~\ref{sec3} below; in particular, in each iteration of ESQM$_{\rm e}$, the $g_i$ in \eqref{eq1} is replaced by its affine approximation at a point \emph{extrapolated from} the current iterate.

In this paper, we study the convergence properties of ESQM$_{\rm e}$ and perform numerical experiments to examine its computational efficiency. In particular, we show that the sequence generated by ESQM$_{\rm e}$ clusters at a critical point if the extrapolation parameters are uniformly bounded above by a certain threshold, under a set of constraint qualifications similarly used in \cite{Ausleder13}. We also construct a suitable potential function and establish the convergence of the whole sequence and its convergence rate by assuming Kurdyka-{\L}ojasiewicz (KL) property of the potential function and additional differentiability conditions on $P_2$ and each $g_i$ in \eqref{eq1}. Furthermore, when $P_2 \equiv 0$ and each $g_i$ is convex, we show that linear convergence can also be established if a certain exact penalty function of \eqref{eq1} is known to be a KL function with exponent $\frac12$. We also discuss how the KL exponent of such an exact penalty function can be derived from that of the function $P+\delta_{C\cap \mathscr{F}}$ from \eqref{eq1} (see Section~\ref{sec2} for notation). Finally, we perform numerical experiments on compressed sensing models with different types of measurement noises taking the form of \eqref{eq1}. Our experiments on random instances illustrate the empirical acceleration of ESQM$_{\rm e}$ over a basic version of ESQM, and also suggest that ESQM$_{\rm e}$ outperforms the natural competing algorithm SCP$_{\rm ls}$ from \cite{yu21}.

The remainder of the paper is organized as follows. We present notation and preliminary materials in Section~\ref{sec2}. Our algorithm, ESQM$_{\rm e}$ is presented in Section~\ref{sec3}, and its subsequential and sequential convergences are established in Section~\ref{sec41}. We discuss the convergence behavior in the convex setting (i.e., $P_2 \equiv 0$ and each $g_i$ is convex in \eqref{eq1}) in Section~\ref{sec42}, and the relationship between the KL exponent of the function $P+\delta_{C\cap \mathscr{F}}$ from \eqref{eq1} and that of the exact penalty function used in the analysis in Section~\ref{sec42} is studied in Section~\ref{sec5}. Numerical experiments are presented in Section~\ref{sec6}.

\section{Notation and preliminaries}\label{sec2}

In this paper, we let $\mathbb{R}$ and $\mathbb{R}_+$ denote the sets of real numbers and nonnegative real numbers respectively, and $\mathbb{N}$ is the set of positive integers. We also let $\mathbb{R}^n$ and $\mathbb{R}_+^n$ denote the $n$-dimensional Euclidean space and its nonnegative orthant respectively. For an $x\in \R$, we let $(x)_+$ denote $\max\{x,0\}$. For an $x\in \mathbb{R}^n$, we let $\|x\|$ denote its Euclidean norm; moreover, for $x$ and $y\in \mathbb{R}^n$, we let $\langle x,y \rangle$ denote their inner product.

For an extended-real-valued function $f:\mathbb{R}^n \to (-\infty,+\infty]$, we say that $f$ is proper if $\dom f := \{ x: f(x)<\infty \}\neq \emptyset$. A proper function $f$ is said to be closed if it is lower semicontinuous. We use $x^k\overset{f}{\to} x$ to denote $x^k\to x$ and $f(x^k)\to f(x)$. For a proper closed function $f$, the regular subdifferential of $f$ at $w\in \dom f$ is given by
$$
\widehat{\partial} f(w):= \left\{ \xi\in\mathbb{R}^n :\liminf_{v\to w,v\neq w} \frac{f(v) - f(w) - \langle \xi, v - w \rangle }{\| v - w\|}\geq 0 \right\}.
$$
The (limiting) subdifferential of $f$ at $w\in \dom f$ is given by
\begin{equation*}
\partial f(w):= \left\{ \xi\in\mathbb{R}^n:\exists w^{k}\overset{f}{\to} w, \xi^{k}\to\xi \text{ with } \xi^{k} \in \widehat{\partial} f(w^{k}) \text{ for each } k\right\},
\end{equation*}
and we set $\partial f(x)=\widehat{\partial}f(x) = \emptyset$ when $ x\notin \dom f$. We also define $\dom \partial f := \{x\in \R^n:\partial f(x)\neq \emptyset\}$. The above subdifferential of $f$ is consistent with the classical subdifferential of $f$ when $f$ is in addition convex; indeed, in this case, we have
$$
\partial f(w)=\left\{ \xi\in\mathbb{R}^n:\langle \xi, v - w \rangle\leq f(v) - f(w) ~~ \forall v\in \mathbb{R}^n \right\};
$$
see, for example, \cite[proposition 8.12]{rock97a}.
For a nonempty closed set $D\subseteq \mathbb{R}^n$, the indicator function $\delta_D$ is defined by
\begin{equation*}
\delta_D(x) = \left\{
\begin{array}{lr}
 0~~ &x\in D,
 \\  \infty ~~ &x\notin D.
 \end{array}
 \right.
\end{equation*}
The normal cone of $D$ at $x\in D$ is defined by
$\mathcal{N}_D(x) :=\partial \delta_D(x)$.
Finally, the distance from a point $x$ to $D$ is denoted by $\d(x, D)$, and the convex hull of $D$ is denoted by $\conv D$.

We next recall some important definitions that will be used in the sequel. We start by recalling the following constraint qualification for \eqref{eq1} (which was also used in \cite{Ausleder13}), and the (associated) first-order optimality conditions for \eqref{eq1}.
\begin{definition}[{{\bf RCQ}}]\label{RCQ}
We say that the Robinson constraint qualification holds at an $x\in\R^n$ for \eqref{eq1} if the following statement holds:
$$ RCQ(x):~\exists y\in C \text{ such that } g_i(x) + \langle\nabla g_i(x), y-x\rangle < 0~~ \forall i = 1,\ldots m.$$
\end{definition}

\begin{definition}[{{\bf Critical point}}]\label{Stationary}
For \eqref{eq1}, we say that $x$ is a critical point of \eqref{eq1} if $x\in C$ and there exists $\lambda=(\lambda_1, \lambda_2, \dots, \lambda_m)\in \mathbb{R}_+^m$ such that $(x, \lambda)$ satisfies the following conditions:
\begin{enumerate}[{\rm (i)}]
    \item $ g_i(x)\leq 0 ~~\forall i=1,\dots,m,$
    \item $ \lambda_i g_i(x)=0 ~~\forall i=1,\dots,m,$
    \item $0\in\partial P_1(x) - \partial P_2(x) + \sum_{i=1}^{m}\lambda_i\nabla g_i(x) + \mathcal{N}_C(x).$
\end{enumerate}	
\end{definition}

One can show using similar arguments as in \cite[Section 2]{yu21} that if $RCQ(x)$ holds at every $x\in C\cap \mathscr{F}$, then any local minimizer of \eqref{eq1} is a critical point of \eqref{eq1}.

Next, we recall the definitions of Kurdyka-{\L}ojasiewicz (KL) property and exponent.
\begin{definition}[{{\bf Kurdyka-{\L}ojasiewicz (KL) property and exponent}}]\label{KLd}
A proper closed function $f$ is said to satisfy the KL property at $\bar{x}\in \dom\partial f$ if there exist $r\in (0,\infty]$, a neighborhood $U$ of $\bar{x}$, and a continuous concave function $\phi:[0,r)\to \mathbb{R}_+$ satisfying $\phi(0)=0$ such that:
\begin{enumerate}[{\rm (i)}]
    \item $\phi$ is continuously differentiable on $(0,r)$ with $\phi'>0$;
    \item for all $x\in U$ with $f(\bar{x})< f(x) < f(\bar{x}) + r$, it holds that
        \begin{align}\label{eq21}
        \phi'(f(x) - f(\bar{x}))\d(0,\partial f(x))\geq 1.
        \end{align}
\end{enumerate}
If $f$ satisfies the KL property at $\bar{x}\in \dom\partial f$ and the $\phi$ in \eqref{eq21} can be chosen as $\phi(\varsigma)= \rho \varsigma^{1-\alpha}$ for some $\rho>0$ and $\alpha\in[0,1)$, then we say that $f$ satisfies the KL property with exponent $\alpha$ at $\bar{x}$.

A proper closed function $f$ satisfying the KL property at every point in $\dom\partial f$ is called a KL function. A proper closed function $f$ satisfying the KL property with exponent $\alpha\in[0,1)$ at every point in $\dom\partial f$ is called a KL function with exponent $\alpha$.
\end{definition}

Many functions are known to satisfy the KL property. For instance, proper closed semi-algebraic functions satisfy the KL property with some exponent $\alpha\in [0,1)$; see \cite{bolte07_2}. The KL property plays an important role in the global convergence analysis of first order methods and the exponent is important in establishing convergence rates; see, for example, \cite{bolte14,attouch13,attouch10,li18}.

Finally, before ending this section, we recall two technical lemmas. The first lemma concerns the uniformized KL property (see \cite[Section~3.5]{bolte14}) and is taken from \cite[Lemma~3.10]{yu21}. The second lemma is a special case of Robinson \cite{Rob75} concerning error bounds for convex functions, which will be used in Section~\ref{sec5} for studying the KL property of a penalty function associated with \eqref{eq1}.
\begin{lemma}\label{KLinequ}
Let $f:\R^n\rightarrow (-\infty,+\infty]$ be a level-bounded proper closed convex function with $\Lambda:= \Argmin f\not=\emptyset$. Let $\underline{f}:=\inf f$. Suppose that $f$ satisfies the KL property at each point in $\Lambda$ with exponent $\alpha\in[0,1)$. Then there exist $\epsilon >0$, $r_0>0$ and $c_0>0$ such that
\[
\d(x, \Lambda)\leq c_0(f(x) - \underline{f})^{1-\alpha}
\]
for any $x\in\dom \partial f$ satisfying $\d(x, \Lambda)\leq \epsilon$ and $\underline{f}\leq f(x) < \underline{f} + r_0$.
\end{lemma}

\begin{lemma}\label{RobEB}
Let $h:\R^n\to \R^m$ with each component function $h_i$ being convex. Let $\Omega := \{x\in \R^n:\; 0 \in h(x) + \R^m_+\}$ and suppose there exist $x^s\in \Omega$ and $\delta_0 > 0$ such that $\{x\in \R^n:\; \|x\|\le \delta_0\}\subseteq h(x^s) + \R^m_+$.
Then
\[
\d(x,\Omega)\leq \frac{\|x - x^s\|}{\delta_0}\d(0, h(x) + \R^m_+)~~~~~  \forall x\in \R^n.
\]
\end{lemma}

\section{Algorithmic framework}\label{sec3}

In this section, we present our algorithm for solving \eqref{eq1}. To describe our algorithm, following the discussion in \cite[Section~3]{wen17}, for each $i$, notice that we can rewrite $g_i$ (whose gradient is Lipschitz continuous) as $g_i = g_i^1 - g_i^2$, where $g_i^1$ and $g_i^2$ are two convex functions with Lipschitz continuous gradients. The next remark concerns the Lipschitz continuity moduli of $\nabla g_i$, $\nabla g_i^1$ and $\nabla g_i^2$.

\begin{remark}[Lipschitz continuity moduli]\label{Remarkg}
Here and throughout, we denote a Lipschitz continuity modulus of $\nabla g_i^1$ by $L_{g_i} > 0$ and a Lipschitz continuity modulus of $\nabla g_i^2$ by $\ell_{g_i} \ge 0$. In addition, by taking a larger $L_{g_i}$ if necessary, we will assume without loss of generality that $L_{g_i} \geq \ell_{g_i}$. Then one can show that $\nabla g_i$ is Lipschitz continuous with a modulus $L_{g_i}$. We also define $L_g := \max\{L_{g_i}: i=1,\dots,m\}$ and $\ell_g = \max\{\ell_{g_i}: i=1,\dots,m\}$.
\end{remark}

The algorithm we study in this paper is presented as Algorithm \ref{alg:Framwork} below; here and throughout, for notational simplicity, for each $u$, $w\in \R^n$, we define
\begin{equation}\label{ling}
\ling{i}{u}{w} := g_i(w) + \langle\nabla g_i(w), u - w\rangle \ \ \ \forall i = 1,\ldots,m, \ \ {\rm and}\ \ \ling{0}{u}{w}:= 0.
\end{equation}
We identify our algorithm as an extended sequential quadratic method with extrapolation (ESQM$_{\text{e}}$), where ``extrapolation" refers to \eqref{defyk}. This is because when $\beta_k \equiv 0$, our algorithm reduces to an instance of the ESQM proposed in \cite{Ausleder13}, whose convergence was studied for solving \eqref{eq1} when the $P$ and each $g_i$ in \eqref{eq1} are in addition twice continuously differentiable.\footnote{More precisely, when the $P$ in \eqref{eq1} is smooth with Lipschitz gradient (say, with modulus $L_P$) and $\beta_k \equiv 0$, our algorithm applied to \eqref{eq1} with $P_1(x) := \frac{L_P}2\|x\|^2$ and $P_2(x) := \frac{L_P}2\|x\|^2 - P(x)$ becomes an instance of the ESQM in \cite{Ausleder13}.} Notice that $(x^{k+1},s^{k+1})$ solves the subproblem in \eqref{eq2} if and only if $s^{k+1} = \max_{i = 1,\cdots,m}[\ling{i}{x^{k+1}}{y^k}]_+$ and
\begin{equation}\label{subproblem2}
x^{k+1} \in \Argmin\limits_{x\in C}~ P_1(x) - \langle \xi^k, x \rangle + \theta_k \max_{i = 1,\cdots,m}[\ling{i}{x}{y^k}]_+ + \frac{\theta_k L_g}{2}\| x - y^k \|^2.
\end{equation}
Since problem \eqref{subproblem2} has a unique solution as an optimization problem with a nonempty closed convex feasible set and a (real-valued) strongly convex objective, we conclude that the subproblem in \eqref{eq2} has a unique solution. While this subproblem requires an iterative solver in general, we refer the readers to \cite[Appendix~A]{zhang23} for an efficient routine for solving the subproblem in \eqref{eq2} with some specific $P_1$ when $m = 1$.
\begin{algorithm}
\caption{ESQM$_{\text{e}}$ for solving \eqref{eq1}}\label{alg:Framwork}
\begin{algorithmic}
\STATE
\begin{description}
  \item[\bf Step 0.] Choose $x^{-1}=x^0\in C$, $\theta_0>0$, $d>0$, and $\{\beta_k\}\subseteq\left[0,\sqrt{\frac{L_g}{L_g + \ell_g}}~\right)$ with $\bar\beta:= \sup_k\beta_k < \sqrt{\frac{L_g}{L_g + \ell_g}}$, where $L_g = \max\{L_{g_i}:  i=1,\dots,m\}$ and $\ell_g = \max\{\ell_{g_i}: i=1,\dots,m\}$ as in Remark~\ref{Remarkg}. Set $k = 0$.
  \item[\bf Step 1.] Set
    \begin{equation}\label{defyk}
      y^k = x^k + \beta_k(x^k - x^{k-1}).
    \end{equation}
  \item[\bf Step 2.] Take any $\xi^k\in \partial P_2(x^{k})$ and compute
    \begin{align}\label{eq2}
    (x^{k+1},s^{k+1})\in \Argmin\limits_{(x,s)\in \mathbb{R}^{n+1}}\quad &P_1(x) - \langle \xi^k, x \rangle
	+ \theta_k s + \frac{\theta_k L_g}{2}\| x - y^k \|^2   \notag
	\\ \text{s.t.} \quad &\ling{i}{x}{y^k} \leq s, ~~~ i=1,\dots,m,
	\\& (x,s)\in C\times \mathbb{R_+},  \notag
	\end{align}
where ${\rm lin}_{g_i}$ is defined in \eqref{ling}.
  \item[\bf Step 3.]  If $\ling{i}{x^{k+1}}{y^k}\leq 0$ for all $i$, then $\theta_{k+1}=\theta_k$; otherwise $\theta_{k+1}=\theta_k+d$. Update $k\leftarrow k+1$ and go to step 1.
\end{description}
\end{algorithmic}
\end{algorithm}

The convergence properties of our algorithm will be studied in Section~\ref{sec4}, and we end this section by presenting some useful facts concerning the subproblem \eqref{eq2}. The first two items are simple observations already established in the preceding discussions, and they are stated here for easy reference later.
\begin{lemma}\label{subproremarks}
Suppose that $x^k\in C$ is generated at the beginning of the $k$-th iteration of Algorithm \ref{alg:Framwork} for some $k\geq 0$. Then the following statements hold:
\begin{enumerate}[{\rm (i)}]
    \item $s^{k+1} = \max_{i = 1,\cdots,m}[\ling{i}{x^{k+1}}{y^k}]_+$.
    \item Problem \eqref{eq2} has a unique solution.
    \item Let $g_0:= 0$. Then $x^{k+1}$ is a component of the minimizer of the subproblem in \eqref{eq2} if and only if there exist $\lambda_i^k \geq 0$ for all $i\in I_k(x^{k+1})$ such that $\sum_{i\in I_k(x^{k+1})}\lambda_i^k = 1$ and
        \begin{equation*}
        0\in \partial P_1(x^{k+1}) - \xi^k + \theta_k\sum_{i\in I_k(x^{k+1})} \lambda_i^k\nabla g_i(y^k) + \theta_kL_g(x^{k+1} - y^k) + \mathcal{N}_C(x^{k+1}),
        \end{equation*}
        where
        \begin{equation}\label{defiIk}
        I_k(x): = \left\{\iota\in\{0,1,\cdots,m\}: \ling{\iota}{x}{y^k} = \max_{i = 0, 1,\cdots,m}\ling{i}{x}{y^k}\right\}.
        \end{equation}
\end{enumerate}
\end{lemma}

\begin{proof}
Items (i) and (ii) were established in the discussions preceding this lemma.

We now prove (iii). Recall that $x^{k+1}$ is a component of the minimizer of the subproblem in \eqref{eq2} if and only if it is a minimizer of the convex problem \eqref{subproblem2}. Using $g_0\equiv 0$ and \cite[Theorem 23.8]{Ro70}, this is further equivalent to
\begin{align*}
0&\!\in\! \partial P_1(x^{k+1}) \!-\! \xi^k \!+\! \theta_k\partial\!\left(\max_{i = 0, 1,\cdots,m}\{\ling{i}{\cdot}{y^k}\}\right)\!(x^{k+1}) \!+\! \theta_kL_g(x^{k+1} - y^k) \!+\! \mathcal{N}_C(x^{k+1})\\
& \!\!\overset{\rm(a)}=\! \partial P_1(x^{k+1}) - \xi^k + \theta_k\conv\{\nabla g_i(y^k): i\in I_k(x^{k+1})\} + \theta_kL_g(x^{k+1} - y^k) + \mathcal{N}_C(x^{k+1}),
\end{align*}
where (a) follows from \cite[Exercise~8.31]{rock97a} with $I_k(\cdot)$ defined in \eqref{defiIk}.
\end{proof}

\section{Convergence properties}\label{sec4}

\subsection{Convergence analysis for ESQM$_{\rm e}$}\label{sec41}
We first show that the successive changes of the $\{x^k\}$ generated by ESQM$_{\rm e}$ vanish.
\begin{theorem}[Vanishing successive changes]\label{suffdec}
Consider \eqref{eq1} and let $\{(x^k,y^k,\theta_k)\}$ be generated by Algorithm \ref{alg:Framwork}. Then the following statements hold:
\begin{enumerate}[{\rm (i)}]
    \item The sequence $\{x^k\}$ belongs to $C$ and is bounded.
    \item Let $\bar{m} := \inf\{P(x):x\in C\}$. Then $\bar m \in \R$ and for any $k\geq 1$,
        $$
        Q(x^{k+1},x^{k},y^{k},\theta_{k+1}) \leq Q(x^k,x^{k-1},y^{k-1},\theta_k) - \left(1 - \frac{L_g + \ell_g}{L_g}\beta_k^2\right)\frac{L_g}{2}\| x^{k} - x^{k-1}\|^2,
        $$
    where
    \[
    Q(x,y,z,\theta):=\frac{P(x) - \bar{m}}{\theta} + \max_{i = 1,\cdots,m}\left[\ling{i}{x}{z}\right]_++ \frac{L_g}{2}\| x - y \|^2 + \frac{L_g}{2}\| x - z \|^2.
    \]
    \item It holds that $\sum_{k = 1}^{\infty} \frac{L_g - (L_g + \ell_g)\beta_k^2}{2} \| x^k - x^{k-1}\|^2 < \infty$, and $\lim_{k \rightarrow\infty}\|x^k - x^{k-1}\| = 0$ and $\lim_{k \rightarrow\infty}\|x^k - y^k\| = 0$.
\end{enumerate}
\end{theorem}

\begin{proof}
(i): Note that $\{x^k\}\subseteq C$ according to \eqref{eq2}. Since $C$ is compact, $\{x^k\}$ is bounded.

(ii): Notice that the objective in \eqref{subproblem2} is strongly convex with $x^{k+1}$ being its unique minimizer over $C$. Using this, and noting that $s^{k+1} = \max_{i = 1,\cdots,m}\left[\ling{i}{x^{k+1}}{y^k}\right]_+$ (see Lemma~\ref{subproremarks}(i)), we have for any $k\ge 0$ that
\begin{equation}
\begin{split}\label{eq6}
& P_1(x^{k+1}) \!-\! \langle \xi^k,x^{k+1}-x^k\rangle + \theta_k s^{k+1} + \frac{\theta_k L_g}{2}\| x^{k+1} - y^k  \|^2 \\
&= P_1(x^{k+1}) \!-\! \langle \xi^k,x^{k+1}\!-x^k\rangle \!+\! \theta_k \!\max_{i = 1,\cdots,m}\left[\ling{i}{x^{k+1}}{y^k}\right]_+ \!\!\!+\! \frac{\theta_k L_g}{2}\| x^{k+1}\! - y^k  \|^2 \\
&\leq P_1(x^{k}) \!+\! \theta_k \max_{i = 1,\cdots,m}\left[\ling{i}{x^{k}}{y^k}\right]_+ \!+\! \frac{\theta_k L_g}{2}\| x^{k}-y^k \|^2 \!-\! \frac{\theta_k L_g}{2}\| x^{k+1}-x^k \|^2.
\end{split}
\end{equation}
Meanwhile, from Remark~\ref{Remarkg} and the definition of ${\rm lin}_{g_i}$ in \eqref{ling}, we see that whenever $k \ge 1$,
\begin{align}\label{gnonconvex}
&\!\! \max_{i = 1,\cdots,m}\left[\ling{i}{x^{k}}{y^k}\right]_+ \notag\\
&\!\!= \max_{i = 1,\cdots,m}\left[g^1_i(y^k) + \langle\nabla g^1_i(y^k), x^k - y^k \rangle - g^2_i(y^k) - \langle\nabla g^2_i(y^k), x^k - y^k \rangle\right]_+ \notag\\
&\!\!\overset{\rm(a)}{\leq}\!\! \max_{i = 1,\cdots,m}\left[g^1_i(x^{k}) \!-\! g^2_i(x^{k}) \!+\! \frac{\ell_{g_i}}{2}\|x^k \!-\! y^{k}\|^2\right]_+ \!\!\!\!=\!\!\! \max_{i = 1,\cdots,m}\left[g_i(x^{k}) \!+\! \frac{\ell_{g_i}}{2}\|x^k \!-\! y^{k}\|^2\right]_+ \notag\\
&\!\! \overset{\rm(b)}{\leq} \max_{i = 1,\cdots,m}\left[\ling{i}{x^{k}}{y^{k-1}}+ \frac{L_{g_i}}{2}\|x^k - y^{k-1}\|^2 + \frac{\ell_{g_i}}{2}\|x^k - y^{k}\|^2\right]_+ \notag\\
&\!\!\overset{\rm(c)}{\leq} \max_{i = 1,\cdots,m}\left[\ling{i}{x^{k}}{y^{k-1}}\right]_+ + \frac{L_g}{2}\|x^k - y^{k-1}\|^2 + \frac{\ell_{g}}{2}\|x^k - y^{k}\|^2,
\end{align}
where (a) holds because of the convexity of $g^1_i$ and the Lipschitz continuity of $\nabla g^2_i$, (b) follows from the Lipschitz continuity of $\nabla g_i$, and (c) holds because $L_g=\max\{L_{g_i}:\;i=1,\dots,m\}$ and $\ell_g=\max\{\ell_{g_i}:\;i=1,\dots,m\}$.
Then, we obtain that when $k \ge 1$,
\begin{align*}
&P(x^{k+1}) = P_1(x^{k+1}) - P_2(x^{k+1})\overset{\rm(a)}{\leq} P_1(x^{k+1}) - \langle\xi^k, x^{k+1} - x^k\rangle - P_2(x^{k})\\
& = P_1(x^{k+1}) - \langle \xi^k, x^{k+1} - x^k\rangle + \frac{\theta_k L_g}{2}\| x^{k+1} - y^k \|^2 - \frac{\theta_k L_g}{2}\| x^{k+1} - y^k\|^2 - P_2(x^{k})\\
&\overset{\rm(b)}{\leq} P_1(x^{k}) + \frac{\theta_k L_g}{2}\| x^{k}-y^k \|^2 - \theta_k s^{k+1} + \theta_k \max_{i = 1,\cdots,m}\left[\ling{i}{x^{k}}{y^{k}}\right]_+ \\
&~~~~~~ - \frac{\theta_k L_g}{2}\| x^{k+1} - x^k\|^2 - \frac{\theta_k L_g}{2}\| x^{k+1} - y^k\|^2 - P_2(x^{k}) \\
&\overset{\rm(c)}{\leq} P(x^{k}) + \theta_k \left(\max_{i = 1,\cdots,m}\left[\ling{i}{x^{k}}{y^{k-1}}\right]_+
 + \frac{L_g}{2}\| x^k - y^{k-1}\|^2 + \frac{\ell_{g}}{2}\|x^k - y^{k}\|^2\right)\\
  & ~~~~~  + \frac{\theta_k L_g}{2}\| x^{k} - y^k\|^2 - \theta_k s^{k+1} - \frac{\theta_k L_g}{2}\| x^{k+1} - x^k\|^2 - \frac{\theta_k L_g}{2}\| x^{k+1} - y^k\|^2,
\end{align*}
where (a) holds because $P_2$ is convex and $\xi^k\in \partial P_2(x^k)$, (b) holds thanks to \eqref{eq6}, and (c) holds because of \eqref{gnonconvex}.

Rearranging terms in the above display and noting that $y^k - x^k = \beta_k(x^k - x^{k - 1})$ for $k \ge 0$ (thanks to the definition of $y^k$ in \eqref{defyk}), we have that for $k\ge 1$,
\begin{align}\label{eq8}
&P(x^{k+1}) + \theta_k s^{k+1} + \frac{\theta_k L_g}{2}\| x^{k+1} - x^k\|^2 + \frac{\theta_k L_g}{2}\| x^{k+1} - y^k\|^2 \notag\\
&\leq P(x^{k}) + \theta_k\max_{i = 1,\cdots,m}\left[\ling{i}{x^{k}}{y^{k-1}}\right]_+ + \frac{\theta_k L_g}{2}\| x^k - y^{k-1}\|^2 \notag\\
&~~~~ + \frac{\theta_k (L_g + \ell_g)}{2}\beta_k^2\| x^{k} - x^{k-1}\|^2 \notag\\
&= P(x^{k}) + \theta_k\max_{i = 1,\cdots,m}[\ling{i}{x^{k}}{y^{k-1}}]_+ + \frac{\theta_k L_g}{2}\| x^{k} - x^{k-1}\|^2 \notag\\
&~~~~~~ + \frac{\theta_k L_g}{2}\| x^k - y^{k-1}\|^2 - \left(1 - \frac{L_g + \ell_g}{L_g}\beta_k^2\right)\frac{\theta_k L_g}{2}\| x^{k} - x^{k-1}\|^2.
\end{align}

Since $P$ is continuous and $C$ is a nonempty compact set, we see that $\bar{m} = \inf\{P(x):x\in C\} \in \R$. Then we can deduce from the definition of $Q$ and the observation $s^{k+1} = \max_{i = 1,\cdots,m}\left[\ling{i}{x^{k+1}}{y^{k}}\right]_+$ (thanks to Lemma~\ref{subproremarks}(i)) that whenever $k\ge 1$,
\begin{align}\label{eq81}
&Q(x^{k+1},x^{k},y^{k},\theta_{k+1}) \notag\\
&= \frac{P(x^{k+1}) - \bar{m}}{\theta_{k+1}} + s^{k+1} + \frac{L_g}{2}\| x^{k+1} - x^{k} \|^2 + \frac{L_g}{2}\| x^{k+1} - y^{k} \|^2 \notag\\
&\overset{\rm(a)}{\leq} \frac{P(x^{k+1}) - \bar{m}}{\theta_k} + s^{k+1} + \frac{L_g}{2}\| x^{k+1} - x^{k} \|^2 + \frac{L_g}{2}\| x^{k+1} - y^{k} \|^2 \notag\\
&\overset{\rm(b)}{\leq} \frac{1}{\theta_k}\bigg[ P(x^{k}) - \bar{m} + \theta_k \max_{i = 1,\cdots,m}[\ling{i}{x^{k}}{y^{k-1}}]_+ + \frac{\theta_k L_g}{2}\| x^{k} - x^{k-1}\|^2 \notag\\
&~~~~ + \frac{\theta_k L_g}{2}\|x^k - y^{k-1}\|^2- \left(1 - \frac{L_g + \ell_g}{L_g}\beta_k^2\right)\frac{\theta_k L_g}{2}\| x^{k} - x^{k-1}\|^2\bigg] \notag\\
& = Q(x^k,x^{k-1},y^{k-1},\theta_k) - \left(1 - \frac{L_g + \ell_g}{L_g}\beta_k^2\right)\frac{L_g}{2}\| x^{k} - x^{k-1}\|^2,
\end{align}
where (a) holds because of the definition of $\bar{m}$ and the facts that $x^{k+1}\in C$ and $\{\theta_k^{-1}\}$ is nonincreasing, and (b) follows from \eqref{eq8} and the fact that $\frac{1}{\theta_k} > 0$.

(iii): Observe that, for any $k\geq 0$,
\begin{equation*}
\begin{split}
&Q(x^{k+1},x^{k},y^k,\theta_{k+1})  \\
&= \frac{P(x^{k+1}) - \bar{m}}{\theta_{k+1}} \!+\! \max_{i = 1,\cdots,m}[\ling{i}{x^{k+1}}{y^k}]_+ \!+\! \frac{L_g}{2}\| x^{k+1} - x^{k} \|^2 \!+\! \frac{L_g}{2}\| x^{k+1}-y^{k} \|^2 \!\geq\! 0.
\end{split}
\end{equation*}
Combining the above display with item (ii), we have
\begin{align*}
&\sum_{k=1}^{\infty}\left(1 - \frac{L_g + \ell_g}{L_g}\beta_k^2\right)\frac{L_g}{2}\| x^{k} - x^{k-1} \|^2 \\
&\leq Q(x^1,x^{0},y^{0},\theta_1) - \liminf_{k\to \infty} Q(x^{k+1},x^{k},y^k,\theta_{k+1})\leq Q(x^1,x^{0},y^{0},\theta_1) <\infty .
\end{align*}
Finally, since $\sup_k\beta_k < \sqrt{\frac{L_g}{L_g + \ell_g}}$, we can deduce from the above display that
$$\lim_{k\to \infty}\| x^k - x^{k-1}\|=0.$$
Combining this with the definition of $y^k$ in \eqref{defyk}, we can obtain further that
$
\lim_{k\to \infty}\| y^k - x^k\| = \lim_{k\to \infty}\beta_k\| x^k - x^{k-1}\| =0$.
\end{proof}

Next, we recall the following assumption involving the RCQ in Definition~\ref{RCQ}. This assumption was first introduced in \cite[Assumption~(A1)]{Ausleder13} for studying ESQM.

\begin{assumption}\label{A1}
For \eqref{eq1}, the $RCQ(x)$ holds at every $x\in C\cap \mathscr{F}$, and for every $x\in C\setminus \mathscr{F}$, there cannot exist $u_i$, $i\in I(x)$, such that
\begin{equation}\label{A11}
u_i\geq 0 ~~ \forall i\in I(x), ~~ \sum_{i\in I(x)}u_i=1, ~~ \left\langle\sum_{i\in I(x)}u_i\nabla g_i(x), z - x\right\rangle\geq0 ~~\forall z\in C,
\end{equation}
where $I(x) := \Big\{ \iota\in\{1, \dots, m\}: g_{\iota}(x) = \max_{i = 1, \dots, m} [g_i(x)]_+ \Big\}$.\footnote{We would like to point out that while our definition of $I(x)$ seems to look slightly different from the corresponding definition, namely $T(x)$, in \cite{Ausleder13} (see the discussions before \cite[Eq.~(6)]{Ausleder13}), one can check that the two definitions are equivalent.}
\end{assumption}
\begin{remark}\label{remarkRCQ}
\begin{enumerate}[{\rm (i)}]
  \item Using \cite[Remark~2.1]{Ausleder13}, one can deduce that if Assumption~\ref{A1} holds, then for any $x\in C$, there cannot exist $u_i$, $i\in I(x)$, such that \eqref{A11} holds.
  \item From \cite[Remark~2.2]{Ausleder13}, we know that if the $RCQ(x)$ holds at every $x\in C$, then Assumption~\ref{A1} holds.
\end{enumerate}
\end{remark}
Using Assumption~\ref{A1} and Theorem~\ref{suffdec}, we will prove in the next theorem that the sequence $\{\theta_k\}$ in Algorithm~\ref{alg:Framwork} is bounded. The same conclusion was established for ESQM in \cite[Theorem~3.1(b)]{Ausleder13}.
\begin{theorem}[Boundedness of $\{\theta_k\}$]\label{alpha}
Consider \eqref{eq1} and suppose that Assumption~\ref{A1} holds. Let $\{(s^k,\theta_k)\}$ be generated by Algorithm~\ref{alg:Framwork}, $\A := \{k\in \mathbb{N}:\theta_{k+1}>\theta_k\}$, and let $|\A|$ denote the cardinality of $\A$. Then $|\A|$ is finite, i.e., there exists $N_0\in \mathbb{N}$ such that $\theta_k \equiv \theta_{N_0}$ whenever $k\geq N_0$. Moreover, $s^{k+1} = 0$ whenever $k\geq N_0$.
\end{theorem}

\begin{proof}
Suppose to the contrary that $|\A| = \infty$. Then by the definition of $\theta_k$ in Step~3 of Algorithm~\ref{alg:Framwork}, we have $\lim_{k\to \infty}\theta_k = \infty$ and $\lim_{k\to \infty}\theta_k^{-1}=0$.
	
We first claim that for each $i$, there exists $n_i\in \mathbb{N}$ such that for all $k\geq n_i$,
\begin{equation*}
 g_i(y^k) + \langle \nabla g_i(y^k),x^{k+1} - y^k \rangle\leq 0,
\end{equation*}
where $\{(x^k,y^k)\}$ is generated by Algorithm~\ref{alg:Framwork}.

Suppose not. Then there exists $i_0\in\{1,\dots, m\}$ and (infinite) subsequences $\{x^{k_j}\}$ and $\{y^{k_j}\}$ such that
\begin{equation*}
 g_{i_0}(y^{k_j}) + \langle \nabla g_{i_0}(y^{k_j}),x^{{k_j}+1} - y^{k_j} \rangle > 0~~~~ \forall j.
\end{equation*}
Using this and recalling the definition of $I_k(\cdot)$ in \eqref{defiIk}, we have that 
\begin{equation*}
 \ling{i}{x^{k_j+1}}{y^{k_j}} > 0~~~ \forall i\in I_{k_j}(x^{k_j+1}),~ \forall j.
\end{equation*}
In particular, $0 \notin  I_{k_j}(x^{k_j+1})$ (see \eqref{ling}). Now, in view of the finiteness of $\left\{I_{k_j}(x^{k_j+1})\right\}$ (since $I_{k_j}(x^{k_j+1})\subseteq \{1, \dots, m\}$ for all $j$), by passing to a further subsequence if necessary, we deduce that there exists a nonempty subset $I_0\subseteq \{1,\dots,m\}$ such that
$I_{k_j}(x^{k_j+1})\equiv I_0$ for all $j$. That is, for all $i\in I_0$,
\begin{equation}\label{eq10}
\ling{i}{x^{k_j+1}}{y^{k_j}} = \max_{l = 0, 1, \dots, m} \left\{ \ling{l}{x^{k_j+1}}{y^{k_j}}\right\} > 0~~ \forall j.
\end{equation}
In addition, from Lemma~\ref{subproremarks}(iii), we have that for each $k_j$, there exist $\lambda_i^{k_j} \geq 0$ for each $i\in I_{k_j}(x^{k_j + 1}) \equiv I_0$, such that $\sum_{i\in I_0}\lambda_i^{k_j} = 1$ and
\begin{align}\label{eq11}
0\in \theta_{k_j}^{-1}(\partial P_1(x^{k_j+1}) - \xi^{k_j}) + L_g(x^{k_j+1} - y^{k_j}) \!+\! \sum_{i\in I_0} \lambda_i^{k_j} \nabla g_i(y^{k_j}) \!+\! \mathcal{N}_C(x^{k_j+1}).
\end{align}
Now, since the sequences $\{x^k\}\subseteq C$ and $\{\lambda_i^{k_j}\}$ (for each $i \in I_0$) are bounded, by passing to a further subsequence if necessary, we assume that $\lim_{j \to \infty} x^{k_j} = x^{*}$ for some $x^*$ and that for each $i\in I_0$, $\lim_{j \to \infty}\lambda_i^{k_j}= \bar{\lambda}_i$ for some $\bar \lambda_i$. Then $x^*\in C$, $\bar{\lambda}_i \ge 0$ (for each $i\in I_0$), $\sum_{i\in I_0} \bar{\lambda}_i = 1$ and $I_0 \subseteq \left\{\iota \in\{0,1,\cdots,m\}: g_{\iota}(x^*) = \max_{i=0,1,\cdots,m} g_i(x^*)\right\}$ (thanks to \eqref{eq10}, \eqref{ling} and Theorem~\ref{suffdec}(iii), and recall that $g_0 \equiv 0$ from Lemma~\ref{subproremarks}(iii)). Since $0\notin I_0$, we see that
\[
I_0 \subseteq I(x^*) =\left\{\iota \in\{1,\cdots,m\}: g_{\iota}(x^*) = \max_{i=1,\cdots,m} [g_i(x^*)]_+\right\},
\]
where $I(x)$ was defined in Assumption~\ref{A1}.
Passing to the limit in \eqref{eq11}, and noting that $\lim_{j \to \infty} \theta_{k_j}^{-1} = 0$, $\lim_{k \to \infty}\| x^{k+1} - y^{k} \|=\lim_{k \to \infty}\| x^{k+1} - x^{k} \|=0$ (thanks to Theorem~\ref{suffdec}(iii)) and the fact that $\{\partial P_1(x^{k_j+1})\}$ and $\{\xi^{k_j}\}$ are uniformly bounded (thanks to the real-valuedness and convexity of $P_1$, $P_2$, the compactness of $C$ and \cite[Theorem~24.7]{Ro70}), we have upon invoking the closedness of $x\mapsto {\cal N}_C(x)$ that
$$0\in \sum_{i\in I_0}\bar{\lambda}_i \nabla g_i(x^*) + \mathcal{N}_C(x^*),$$
which implies that
\begin{align*}
\left\langle \sum_{i\in I_0}\bar{\lambda}_i \nabla g_i(x^*), x-x^* \right\rangle \geq 0 ~~~~ \forall x\in C.
\end{align*}
Since $I_0\subseteq I(x^*)$, this contradicts Assumption~\ref{A1} in view of Remark~\ref{remarkRCQ}(i).

Therefore, if $|\A| = \infty$, then it must hold that for each $i$, there exists $n_i\in \mathbb{N}$, such that for any $k\geq n_i$,
\[
g_i(y^k) + \langle \nabla g_i(y^k),x^{k+1} - y^k \rangle\leq 0.
\]
Let $N_* :=  \max_{i = 1, \dots, m} n_i$. Then for all $i\in\{1, \dots, m\}$ and for any $k\ge N_*$, we have
\[
g_i(y^k) + \langle \nabla g_i(y^k),x^{k+1} - y^k \rangle\leq 0.
\]
In view of this and the definition of $\theta_k$ in Step 3 of Algorithm \ref{alg:Framwork}, we must have $\theta_k\equiv \theta_{N_*}$ for all $k\geq N_*$, which contradicts $\theta_k\rightarrow\infty$. Thus, it must hold that $|\A|< \infty$.

Since $|\A|$ is finite, there exists $N_0\in \mathbb{N}$, such that $\theta_k\equiv\theta_{N_0}$ whenever $k\geq N_0$. From Step 3 of Algorithm \ref{alg:Framwork}, we known that for each $i$, $ g_i(y^k) + \langle \nabla g_i(y^k), x^{k+1} - y^k \rangle\leq 0$, for all $k\geq N_0$. Then Lemma~\ref{subproremarks}(i) asserts that $s^{k+1}=0$ for any $k\geq N_0$.
\end{proof}

We are now ready to prove that any cluster point of the $\{x^k\}$ generated by Algorithm~\ref{alg:Framwork} is a critical point of \eqref{eq1}.
\begin{theorem}[Subsequential convergence]\label{subconver}
Consider \eqref{eq1} and suppose that Assumption \ref{A1} holds. Let $\{x^k\}$ be generated by Algorithm \ref{alg:Framwork}. Then for any accumulation point $\bar x$ of $\{x^k\}$, there exists $\bar{\lambda}_i\geq 0$ for each $i\in \tilde{I}(\bar{x})$ such that $\sum_{i\in \tilde{I}(\bar{x})} \bar{\lambda}_i = 1$ and
\begin{equation}\label{critical3333}
        0\in \partial P_1(\bar{x}) - \partial P_2(\bar x) + \theta_{N_0}\sum_{i\in \tilde{I}(\bar{x})} \bar{\lambda}_i\nabla g_i(\bar{x}) + \mathcal{N}_C(\bar{x}),
\end{equation}
where $\tilde{I}(\bar{x}):=\left\{\iota \in\{0,1,\cdots,m\}: g_{\iota}(\bar{x}) = \max_{i=0,1,\cdots,m} \{g_i(\bar x)\}\right\}$, $g_0 = 0$, and $\theta_{N_0}$ is defined in Theorem~\ref{alpha}; moreover, $\bar x$ is a critical point of \eqref{eq1}.
\end{theorem}
\begin{proof}
Suppose that $\bar{x}$ is an accumulation point of $\{x^k\}$ with $\lim_{j\to \infty} x^{k_j} = \bar{x}$ for some convergent subsequence $\{x^{k_j}\}$. Let $\{\xi^k\}$ be generated in Algorithm~\ref{alg:Framwork} and $\{\lambda^k_i\}$ with $i\in I_k(x^{k+1})$ be as in Lemma~\ref{subproremarks}(iii). Then, in view of the finiteness of $\{I_{k_j}(x^{k_j+1})\}$ (since $I_{k_j}(x^{k_j+1})\subseteq \{0, 1, \dots, m\}$ for all $j$), by passing to a further subsequence if necessary, we see that there exists a nonempty subset $I_0\subseteq \{0, 1,\dots,m\}$ such that $I_{k_j}(x^{k_j+1})\equiv I_0$.
Moreover, $\{\lambda^{k_j}_i\}$ for each $i \in I_{k_j}(x^{k_j+1})\equiv I_0$ is bounded as sequences of nonnegative numbers at most $1$, and $\{\xi^k\}$ is bounded thanks to the real-valuedness and convexity of $P_2$ and \cite[Theorem~24.7]{Ro70}. Passing to a further subsequence if necessary, we assume without loss of generality that $\lim_{j\to \infty} \lambda_i^{k_j} = \bar{\lambda}_i\geq 0$ for each $i\in I_0$ and $\lim_{j\to \infty} \xi^{k_j} = \bar{\xi}$; moreover, the property of $\{\lambda^{k_j}_i\}$ with $i\in I_{k_j}(x^{k_j+1})\equiv I_0$ guaranteed by Lemma~\ref{subproremarks}(iii) asserts that for all $j$, it holds that
\begin{equation}\label{inclusion}
\begin{aligned}
&0\in \partial P_1(x^{k_j+1}) \!-\! \xi^{k_j} \!+\! \theta_{k_j}L_g(x^{{k_j}+1} - y^{k_j}) \!+\! \theta_{k_j}\! \sum_{i\in I_0} \lambda_i^{k_j}\nabla g_i(y^{k_j}) \!+\! \mathcal{N}_C(x^{{k_j}+1})\\
&{\rm and}\ \ \sum_{i\in I_0}\lambda^{k_j}_i = 1,\ \ \ \lambda^{k_j}_{i} \ge 0\ \ \forall i \in I_{k_j}(x^{k_j+1})\equiv I_0.
\end{aligned}
\end{equation}
In addition, in view of \eqref{eq2}, we obtain that for each $j$,
\begin{equation}\label{giineq}
g_i(y^{k_j}) + \langle\nabla g_i(y^{k_j}), x^{{k_j}+1} - y^{k_j} \rangle \leq s^{{k_j}+1} ~~~~ \forall i =1,\dots,m.
\end{equation}

Now, note that $\lim_{k\to \infty} \| x^{k} - x^{k-1} \|=\lim_{k\to \infty} \| x^{k+1} - y^{k} \|=0$ (thanks to Theorem~\ref{suffdec}(iii)), $s^{k_j+1} = 0$ and $\theta_{k_j} \equiv \theta_{N_0}$ whenever $k_j\geq N_0$ (thanks to Theorem~\ref{alpha}). Passing to the limit in \eqref{giineq} and \eqref{inclusion}, we see that
\begin{equation}\label{critical1}
g_i(\bar{x})\le 0 ~~ \forall i=1,\dots,m,\ \ \sum_{i\in I_0}\bar \lambda_i = 1,\ \ \ \bar\lambda_i \ge 0 \ \ \forall i\in I_0,
\end{equation}
and
\begin{equation}\label{critical3}
0\in \partial P_1(\bar{x}) - \partial P_2(\bar x) + \theta_{N_0}\sum_{i\in I_0} \bar{\lambda}_i\nabla g_i(\bar{x}) + \mathcal{N}_C(\bar{x}).
\end{equation}
where we also invoked the closedness of $\partial P_1$, $\partial P_2$ and $\mathcal{N}_C$ to deduce \eqref{critical3}. Furthermore, we have from the definition of $I_{k_j}(x^{k_j + 1})$ in \eqref{defiIk} (and recall that $I_{k_j}(x^{k_j + 1})\equiv I_0$) and Theorem~\ref{suffdec}(iii) that
\begin{equation}\label{I0}
 I_0\subseteq \tilde{I}(\bar{x}):=\left\{\iota \in\{0,1,\cdots,m\}: g_{\iota}(\bar{x}) = \max_{i=0,1,\cdots,m} \{g_i(\bar{x})\}\right\}.
\end{equation}
Then the inclusion \eqref{critical3333} follows from \eqref{critical3} and \eqref{critical1} upon noting $I_0\subseteq \tilde{I}(\bar{x})$ (see \eqref{I0}) and defining $\bar \lambda_i = 0$ for $i\in \tilde I(\bar x)\setminus I_0$.

Finally, let $\hat{\lambda}_i := \theta_{N_0}\bar{\lambda}_i\geq 0$ for all $i\in I_0\cap\{1,\cdots,m\}$, and $\hat{\lambda}_i = 0$ for all $i\in \{1,\cdots,m\}\setminus I_0$. Then by \eqref{critical1} and $I_0\subseteq \tilde{I}(\bar{x})$ (see \eqref{I0}), we have that
\begin{equation}\label{critical2}
\hat{\lambda}_i g_i(\bar{x}) = 0 ~~ \forall i=1,\dots,m;
\end{equation}
indeed, for each $i\in I_0$, we have $g_i(\bar{x}) = 0$, and for each $i\notin I_0$, we have $\hat{\lambda}_i = 0$.

Notice that $\nabla g_0 (\bar{x}) = 0$ (thanks to $g_0 \equiv 0$). Using the definition of $\hat{\lambda}_i$ and \eqref{critical3}, we have
\begin{equation}\label{critical33}
\begin{aligned}
0&\in \partial P_1(\bar{x}) - \partial P_2(\bar{x}) + \sum_{i=1}^m \hat{\lambda}_i\nabla g_i(\bar{x}) + \mathcal{N}_C(\bar{x}).
\end{aligned}
\end{equation}
Combining \eqref{critical1}, \eqref{critical2}, \eqref{critical33} and the above definition of $\hat\lambda$, we conclude that $\bar{x}$ is a critical point of \eqref{eq1}.
\end{proof}


We next derive the global convergence property of the $\{x^k\}$ generated by Algorithm~\ref{alg:Framwork}. We will need to make use of the following function,
\begin{equation}\label{defH}
H(x,y,z) \!:=\! \frac{P(x) - \bar{m}}{\hat{\theta}} +\! \max_{i = 1,\ldots,m}[\ling{i}{x}{z}]_+\! + \frac{L_g}{2}\| x-y \|^2 + \frac{L_g}{2}\| x - z \|^2 + \delta_{C}(x),
\end{equation}
where $\bar{m}$ is defined in Theorem~\ref{suffdec}(ii), and $\hat{\theta}:= \theta_{N_0}$ with $N_0$ defined in Theorem~\ref{alpha}. Our analysis follows the nowadays standard convergence arguments based on Kurdyka-{\L}ojasiewicz property; see, for example, \cite{attouch10,attouch13,bolte14}. In essence, under Assumption~\ref{A1}, we will show that $H$ has sufficient descent along the sequence $\{(x^{k+1},x^{k},y^k)\}$ for all sufficiently large $k$, and $H$ is constant on the set of accumulation points of $\{(x^{k+1},x^{k},y^k)\}$. We will also show that $ \d(0,\partial H(x^{k+1},x^{k},y^k))$ is suitably bounded by successive changes of the iterates by imposing additional differentiability assumptions on each $g_i$ and $P_2$. These together with an additional assumption that $H$ satisfies the KL property will be used to establish global convergence of the $\{x^k\}$ generated by Algorithm~\ref{alg:Framwork}.

We start with a remark concerning the sufficient descent property.
\begin{remark}[Sufficient descent]\label{rebarh}
Consider \eqref{eq1} and suppose that Assumption~\ref{A1} holds. Notice from the definition of $Q$ in Theorem~\ref{suffdec}{\rm (ii)} and that of $H$ in \eqref{defH} that $H(x,y,z) = Q(x,y,z,\hat\theta) + \delta_C(x)$. Now, according to Theorem~\ref{alpha}, we have $\theta_k \equiv \theta_{N_0} = \hat\theta$ for all $k\geq N_0$. Thus, we have $H(x^{k},x^{k-1},y^{k-1}) = Q(x^{k},x^{k-1},y^{k-1},\hat{\theta})$ for all $k\geq N_0$, where $\{(x^k,y^k)\}$ is generated by Algorithm \ref{alg:Framwork}. Then one can see that the sequence $\{H(x^{k+1},x^{k},y^k)\}_{k\ge N_0}$ is nonincreasing thanks to Theorem~\ref{suffdec}{\rm (ii)}, and it holds that
\[
H(x^{k+1},x^{k},y^{k}) \leq H(x^{k},x^{k-1},y^{k-1}) - \frac{L_g - (L_g + \ell_g)\bar{\beta}^2}{2}\| x^{k} - x^{k-1}\|^2\ \ \ \ \forall k\ge N_0,
\]
where $\bar{\beta} = \sup_k\beta_k$, and notice that $L_g > (L_g + \ell_g)\bar\beta^2$ thanks to the choice of $\{\beta_k\}$.
\end{remark}

\begin{lemma}\label{lemma1}
Consider \eqref{eq1} and suppose that Assumption~\ref{A1} holds. Let $\{(x^k,y^k)\}$ be generated by Algorithm \ref{alg:Framwork}, $H$ be defined in \eqref{defH}, and $\Omega$ be the set of accumulation points of $\{(x^{k+1},x^{k},y^{k})\}$. Then $\Omega$ is a nonempty compact set, $\omega := \lim_{k\to \infty} H(x^{k+1},x^{k}, y^{k})$ exists, and $H \equiv \omega$ on $\Omega$.
\end{lemma}
\begin{proof}
From Theorem~\ref{suffdec}(i), we have that the set of accumulation points of $\{x^k\}$, denoted by $\varLambda$, is a nonempty compact set. Since $\lim_{k\to \infty} \| x^k - x^{k-1} \|=\lim_{k\to \infty} \| x^k - y^k\| = 0$ thanks to Theorem~\ref{suffdec}(iii), one can see that $\Omega= \{(\bar{x},\bar{x},\bar{x}): \bar{x}\in \varLambda\}$, which is a nonempty compact set.

Next, according to Remark~\ref{rebarh}, the sequence $\{H(x^{k+1},x^{k}, y^{k})\}_{k\ge N_0}$ is nonincreasing. Moreover, one can see from the definition of $H$ (see \eqref{defH}) that $\{H(x^{k+1},x^{k}, y^{k})\}$ is bounded from below (by zero). Thus, $\omega := \lim_{k\to \infty} H(x^{k+1},x^{k}, y^{k})$ exists.

For any $(\bar{x},\bar{x},\bar{x})\in \Omega$, let $\{x^{k_j}\}$ be a convergent subsequence with $\lim_{j\to \infty} x^{k_j} = \bar{x}$.
Since $P$ and each $g_i$ are continuous, and $\lim_{k\to \infty} \| x^k - x^{k-1}\| = \lim_{k\to \infty} \| x^k - y^{k}\| = 0$ (see Theorem~\ref{suffdec}(iii)), we obtain that
\begin{align*}
& H(\bar{x}, \bar{x}, \bar{x}) = \frac{P(\bar{x}) - \bar{m}}{\hat{\theta}} + \max_{i = 1,\cdots,m}\left[\ling{i}{\bar x}{\bar x}\right]_+\\
&\!\!=\!\lim_{j\to \infty}\! \frac{P(x^{k_j+1}) \!-\! \bar{m}}{\hat{\theta}} \!+\!\!\max_{i = 1,\cdots,m}\!\left[\ling{i }{x^{k_j+1}}{y^{k_j}}\right]_+ \!\!\!+\! \frac{L_g}{2}\| x^{k_j+1} \!\!-\! x^{k_j} \|^2 \!+\! \frac{L_g}{2}\| x^{k_j+1} \!\!-\! y^{k_j} \|^2 \\
&\!\!=\lim_{j\to \infty} H(x^{k_j+1},x^{k_j},y^{k_j}) = \lim_{k\to \infty} H(x^{k+1},x^{k}, y^{k}) = \omega.
\end{align*}
Since $(\bar{x}, \bar{x}, \bar{x})\in \Omega$ is arbitrary, we conclude that $H \equiv \omega$ on $\Omega$.
\end{proof}

Next, we introduce an assumption for deriving a bound on $\d(0,\partial H (x^{k+1}, x^k, y^k))$. This assumption was also used in \cite{wen18,yu21} and is satisfied in many applications; see \cite{wen18}.
\begin{assumption}\label{A2}
Each $g_i$ in \eqref{eq1} is twice continuously differentiable. The function $P_2$ is continuously differentiable on an open set $U_0$ containing $\mathcal{X}$, and $\nabla P_2$ is locally Lipschitz continuous on $U_0$, where $\mathcal{X}$ is the set of critical points of \eqref{eq1}.
\end{assumption}

Now, we present the following bound on $\d(0,\partial H(x^{k+1},x^{k},y^{k}))$.
\begin{lemma}\label{th2.1}
Consider \eqref{eq1} and suppose that Assumptions~\ref{A1} and \ref{A2} hold. Let $\{(x^k,y^k)\}$ be generated by Algorithm \ref{alg:Framwork} and $H$ be defined in \eqref{defH}. Then there exist $\tau > 0$ and $N_1\in \mathbb{N}$ such that for all $k \ge N_1$, we have
\[
\d (0,\partial H(x^{k+1},x^{k},y^{k}))\le \tau (\|x^{k+1} - x^{k}\| + \|x^{k} - x^{k-1}\|).
\]
\end{lemma}
\begin{proof}
Let $\varLambda$ be the set of accumulation points of $\{x^k\}$. Then $\varLambda$ is nonempty and compact in view of Theorem~\ref{suffdec}(i), and $\varLambda\subseteq {\cal X}$ thanks to Theorem~\ref{subconver}, where ${\cal X}$ is defined in Assumption~\ref{A2}. Moreover, we have $\lim_{k\to \infty}\d(x^k,\varLambda)=0$. Since $\Lambda\subseteq {\cal X}\subset U_0$ (where $U_0$ is defined in Assumption~\ref{A2}) and $\Lambda$ is compact, there exist a bounded open set $U_1$ and an $N_2\in \mathbb{N}$ such that $x^k\in U_1$ for all $k\geq N_2$ and the closure of $U_1$ is contained in $U_0$.\footnote{The existence of such a $U_1$ can be argued as follows: since $\Lambda$ is compact, there exists $\epsilon > 0$ such that $\{x\in \R^n: \d(x,\Lambda) < \epsilon\}\subseteq U_0$. Then set $U_1 := \{x\in \R^n: \d(x,\Lambda) < \epsilon/2\}$.}

Next, let $N_0$ be defined as in Theorem \ref{alpha}. Since $P_2$ is continuously differentiable on $U_0$ and $x^k\in U_1\subset U_0$ for any $k\geq N_1 := \max\{N_0,N_2\}$, we obtain from \cite[Theorem~8.6]{rock97a} that for any $k \ge N_1$,
\begin{align}\label{eq16}
&\!\!\!\!\!\!\!\!\!\!~~~\partial H(x^{k+1},x^{k},y^{k})\supseteq \widehat{\partial} H(x^{k+1},x^{k},y^{k})\notag \\
&\!\!\!\!\!\!\!\!\!\!~~\overset{\rm(a)}\supseteq\begin{bmatrix}
  \frac{1}{\hat{\theta}}\widehat\partial P(x^{k+1}) + \mathcal{N}_C(x^{k+1}) + L_g(x^{k+1}- x^k) + L_g(x^{k+1}- y^k)\\
  - L_g(x^{k+1} - x^{k})\\
   - L_g(x^{k+1}-y^{k})
\end{bmatrix} + \widehat\partial \Xi (x^{k+1},x^k,y^k)\notag\\
&\!\!\!\!\!\!\!\!\!\!~~\overset{\rm(b)}=\begin{bmatrix}
  \frac{1}{\hat{\theta}}\partial P(x^{k+1})+ \mathcal{N}_C(x^{k+1}) + L_g(x^{k+1}- x^k) + L_g(x^{k+1}- y^k)\\
  - L_g(x^{k+1} - x^{k})\\
   - L_g(x^{k+1}-y^{k})
\end{bmatrix} + \partial \Xi (x^{k+1},x^k,y^k)\notag\\
&\!\!\!\!\!\!\!\!\!\!~~\overset{\rm(c)}{\supseteq}\!\!\!
\left[
\begin{array}{c}
\!\frac{1}{\hat{\theta}}\partial P(x^{k+1}) \!+ \!\!\!\!\sum\limits_{i\in I_k(x^{k+1})}\!\!\! \lambda_i^k \nabla g_i(y^k) \!+\! \mathcal{N}_C(x^{k+1}) \!+\! L_g(x^{k+1} \!-\! x^k) \!+\! L_g(x^{k+1} \!-\! y^k)    \\
 - L_g(x^{k+1} - x^{k}) \\ [5 pt]
\sum_{i\in I_k(x^{k+1})} \lambda_i^k\nabla^2 g_i(y^{k})(x^{k+1} - y^{k}) - L_g(x^{k+1}-y^{k})
\end{array}\!
\right]\!\!,\!\!\!\!
\end{align}
where $\Xi(x,y,z) := \max_{i=1,\ldots,m}[\ling{i}{x}{z}]_+$, and $I_k(x^{k+1})$ and $\lambda_i^k$ are defined as in Lemma~\ref{subproremarks}(iii);
here, (a) holds because of the subdifferential calculus rules in \cite[Proposition~10.5, Corollary~10.9]{rock97a} and the regularity of the normal cone of $C$ in \cite[Theorem~6.9]{rock97a}, (b) holds because $\partial \Xi = \widehat\partial \Xi$ (thanks to \cite[Example~7.28]{rock97a}) and $\partial P(x^{k+1}) = \widehat \partial P(x^{k+1})$ (thanks to the regularity of $P_1$ as asserted in \cite[Proposition~8.12]{rock97a}, the assumption that $P_2$ is continuously differentiable at $x^{k+1}\in U_0$ and the subdifferential calculus rule \cite[Exercise~8.8(c)]{rock97a}), and (c) follows from \cite[Proposition~10.5, Exercise~8.31]{rock97a} and the fact that $\sum_{i\in I_k(x^{k+1})} \lambda_i^k = 1$ and $\lambda_i^k\ge 0$ for all $i\in I_k(x^{k+1})$.

On the other hand, according to Theorem \ref{alpha} and the definition of $\hat\theta$ in \eqref{defH}, we have that $\theta_k \equiv \theta_{N_0} = \hat \theta$ for any $k\ge N_0$. Using this together with the property of $\lambda^k_i$ from Lemma~\ref{subproremarks}(iii) and the differentiability assumption on $P_2$, we obtain that for all $k \ge N_1$,
\begin{align*}
0\in \partial P_1(x^{k+1}) - \nabla P_2(x^{k}) + \hat{\theta}\sum_{i\in I_k(x^{k+1})} \lambda_i^k \nabla g_i(y^k) + \hat{\theta}L_g(x^{k+1} - y^{k})+ \mathcal{N}_C(x^{k+1}).
\end{align*}
Rearranging terms in the above display, we see that
\begin{align}\label{eq23}
\nabla P_2(x^{k}) \!-\! \hat{\theta}\sum_{i\in I_k(x^{k+1})} \lambda_i^k \nabla g_i(y^k) - \hat{\theta}L_g(x^{k+1} - y^{k})\in \partial P_1(x^{k+1}) + \mathcal{N}_C(x^{k+1}).
\end{align}

Since $P_2$ is continuously differentiable in $U_0$ (and hence at $x^k$ and $x^{k+1}$ when $k\ge N_1$), we obtain for any $k\ge N_1$ that
\begin{align}\label{eq22}
&\frac{1}{\hat{\theta}}\left(-\hat{\theta}L_g(x^{k} - y^{k}) + \nabla P_2(x^{k}) - \nabla P_2(x^{k+1})\right) \notag\\
& = \frac{1}{\hat{\theta}}\left(\hat{\theta}L_g(x^{k+1} - x^{k}) - \nabla P_2(x^{k+1}) + \hat{\theta}\sum_{i\in I_k(x^{k+1})} \lambda_i^k\nabla g_i(y^k)\right) \notag\\
&~~~~ + \frac{1}{\hat{\theta}}\left(\nabla P_2(x^{k}) - \hat{\theta}\sum_{i\in I_k(x^{k+1})} \lambda_i^k\nabla g_i(y^k)
 - \hat{\theta}L_g(x^{k+1} - y^{k})\right) \notag\\
& \overset{\rm(a)}{\in}\! \frac{1}{\hat{\theta}}\!\left(\hat{\theta}L_g(x^{k+1} \!-\! x^{k}) \!-\! \nabla P_2(x^{k+1}) + \hat{\theta}\!\!\!\!\sum_{i\in I_k(x^{k+1})}\!\! \lambda_i^k\nabla g_i(y^k)\right) \!+\! \frac{1}{\hat{\theta}}\partial P_1(x^{k+1}) \!+\! \mathcal{N}_C(x^{k+1}) \notag\\
& = \frac{1}{\hat{\theta}}\partial P(x^{k+1}) + \sum_{i\in I_k(x^{k+1})} \lambda_i^k \nabla g_i(y^k) + \mathcal{N}_C(x^{k+1}) + L_g(x^{k+1} - x^k)
\end{align}
where (a) follows from \eqref{eq23}, and the last equality holds thanks to \cite[Exercise~8.8(c)]{rock97a} and the fact that $P = P_1 - P_2$.

Combining \eqref{eq16} and \eqref{eq22}, for any $k\geq N_1$, we have
\begin{equation}
\begin{aligned}
\nonumber
\left[
\begin{array}{c}
\frac{1}{\hat{\theta}}\left(-\hat{\theta}L_g(x^{k} - y^{k}) + \nabla P_2(x^{k}) - \nabla P_2(x^{k+1})\right) + L_g(x^{k+1} - y^{k})\\
-L_g(x^{k+1} - x^{k})\\
\sum_{i\in I_k(x^{k+1})}\lambda_i^k\nabla^2 g_i(y^{k})(x^{k+1} - y^{k}) - L_g(x^{k+1} - y^{k})
\end{array}
\right]
\!\!\in\! \partial H(x^{k+1},x^{k},y^{k}).
\end{aligned}
\end{equation}
Since $\nabla P_2$ is locally Lipschitz continuous on $U_0$ (and hence Lipschitz continuous on the bounded open set $U_1$, say, with modulus $L_{P_2}$), we see for any $k\geq N_1$ that
\begin{align*}
& \d\left(0,\partial H(x^{k+1},x^{k},y^{k})\right)^2 \\
&\leq \left\| \frac{1}{\hat{\theta}}\!\left(\!-\hat{\theta}L_g(x^{k} - y^{k}) \!+\! \nabla P_2(x^{k}) - \nabla P_2(x^{k+1})\right) \!+\! L_g(x^{k+1} \!-\! y^{k})\right\|^2\!\!\!\! +\! \| L_g(x^{k+1} - x^{k})\|^2 \\
&~~~~~~ + \left\| \sum_{i\in I_k(x^{k+1})} \lambda_i^k\nabla^2 g_i(y^k)(x^{k+1} - y^{k}) - L_g(x^{k+1} - y^{k})\right\|^2 \\
&\leq 3L_g^2\| x^{k} - y^{k}\|^2 + \frac{3}{\hat{\theta}^2}L_{P_2}^2\| x^{k+1} - x^{k}\|^2 + 3L_g^2\| x^{k+1} - y^{k}\|^2 + L_g^2\| x^{k+1} - x^{k} \|^2 \\
&~~~~~~ + 2\left\|\sum_{i\in I_k(x^{k+1})} \lambda_i^k\nabla^2 g_i(y^k)\right\|^2 \| x^{k+1}-y^{k} \|^2 + 2L_g^2\| x^{k+1} - y^{k}\|^2.
\end{align*}
The desired conclusion now follows immediately from the above display, the definition and the boundedness of $\{y^k\}$ (thanks to Theorem~\ref{suffdec}(i) and \eqref{defyk}) and the continuity of $\nabla^2 g_i$ (thanks to Assumption~\ref{A2}).
\end{proof}

Now, we present the convergence rate of the $\{x^k\}$ generated by Algorithm~\ref{alg:Framwork} under suitable assumptions. The proof is routine and we refer the readers to, for example, the proofs of Theorems 4.2 and 4.3 of \cite{wen18}.
\begin{theorem}[Global convergence and convergence rate of Algorithm \ref{alg:Framwork} in nonconvex setting]\label{th2.2}
Consider \eqref{eq1}. Suppose that Assumptions~\ref{A1} and \ref{A2} hold, and the $H$ in \eqref{defH} is a KL function. Let $\{(x^k, y^k)\}$ be generated by Algorithm~\ref{alg:Framwork} and $\Omega$ be the set of accumulation points of $\left\{(x^{k+1},x^{k},y^{k})\right\}$. Then $\{x^k\}$ converges to a critical point $\bar{x}$ of \eqref{eq1}. Moreover, if $H$ satisfies the KL property with exponent $\alpha\in [0,1)$ at every point in $\Omega$, then there exists $\underline{N}\in \mathbb{N}$ such that the following statements hold.
\begin{enumerate}[{\rm (i)}]
    \item If $\alpha=0$, then $\{x^k\}$ converges finitely, i.e., $x^k \equiv \bar x$ for all $k> \underline{N}$.
    \item If $\alpha\in (0,\frac{1}{2}]$, then there exist $a_0\in(0, 1)$ and $a_1>0$ such that \[
        \|x^k - \bar{x}\|\leq a_1a_0^k ~~\forall k > \underline{N}.
        \]
    \item If $\alpha\in (\frac{1}{2},1)$, then there exists $a_2>0$ such that
        \[
        \|x^k - \bar{x}\|\leq a_2k^{-\frac{1-\alpha}{2\alpha-1}} ~~\forall k > \underline{N}.
        \]
\end{enumerate}
\end{theorem}

\subsection{Convergence analysis in convex setting}\label{sec42}
We study the convergence properties of Algorithm \ref{alg:Framwork} under the following convex settings.

\begin{assumption}\label{B1}
Suppose that in \eqref{eq1}, $P_2 = 0$ and $g_1,\ldots, g_m$ are convex.\footnote{Under this assumption, we also set $\ell_{g_i}= 0$ for all $i$ in Algorithm~\ref{alg:Framwork}.}
\end{assumption}

\begin{assumption}\label{B2}
The Slater condition holds for $C\cap \mathscr{F}$ in \eqref{eq1}, i.e., there exists $\hat{x}\in C$ with $g_i(\hat{x})<0$ for $i=1,\dots,m$.
\end{assumption}
\begin{remark}\label{rem43}
If each $g_i$ is convex and Assumption~\ref{B2} holds, then $RCQ(x)$ holds at every $x\in C$, which implies that Assumption~\ref{A1} holds thanks to Remark~\ref{remarkRCQ}(ii).
\end{remark}

Now, we present the convergence properties of Algorithm \ref{alg:Framwork} under Assumptions~\ref{B1} and \ref{B2}. Unlike our convergence rate result in Theorem~\ref{th2.2} which was based on the KL property of the function $H$ in \eqref{defH}, our analysis in this section is based on the KL property of the following function:
\begin{equation}\label{definfalpha}
F_{\eta}(x) := \frac{1}{\eta}(P_1(x) - \hat m) + \delta_C(x) + \max_{i = 1,\cdots,m}\left[g_i(x)\right]_+,
\end{equation}
where $\eta > 0$ and $\hat{m} := \inf\{P_1(x): x\in C\}\in \R$. Compared with $H$, the explicit KL exponent of $F_{\eta}$ is generically readily obtainable (from that of $P_1+\delta_{C\cap \mathscr{F}}$), as we will discuss in Section~\ref{sec5}.

\begin{theorem}[Convergence rate of Algorithm \ref{alg:Framwork} in convex setting]\label{Edecres}
Consider \eqref{eq1} and suppose that Assumptions~\ref{B1} and \ref{B2} hold. Let $\{(x^k,\theta_k)\}$ be generated by Algorithm \ref{alg:Framwork}. Then the following statements hold.
\begin{enumerate}[{\rm (i)}]
    \item For any $k\ge 1$,
          \[
          E(x^{k+1},x^{k},\theta_{k+1}) \leq E(x^k,x^{k-1},\theta_k) - \frac{(1 - \beta_k^2)L_g}{2}\|x^{k} - x^{k-1}\|^2,
          \]
          where $E(x,y,\theta):=\frac{1}{\theta}\big(P_1(x) - \hat{m} + \delta_C(x) + \theta\max_{i = 1,\cdots,m}[g_i(x)]_+ + \frac{\theta L_g}{2}\| x - y \|^2\big)$ with $\hat m$ defined as in \eqref{definfalpha}.
    \item Let $\Omega$ be the set of accumulation points of $\left\{(x^{k+1},x^{k},\theta_k)\right\}$. Then $\Omega$ is a nonempty compact set, $\bar{\omega} := \lim_{k\to \infty} E(x^{k+1},x^{k}, \theta_k)$ exists, and $E \equiv \bar{\omega}$ on $\Omega$.
    \item If the function\footnote{i.e., the $F_\eta$ in \eqref{definfalpha} with $\eta = \hat \theta$, where $\hat\theta$ is given in \eqref{defH}.} $F_{\hat\theta}$ is a KL function with exponent $\frac{1}{2}$,
     then $\{x^k\}$ converges to a minimizer $x^*$ of \eqref{eq1}, and there exist $c_0 > 0$, $s\in (0,1)$ and $k_0\in \mathbb{N}$ such that
        \[
        \|x^k - x^*\|\leq c_0 s^k ~~\forall k > k_0.
        \]
\end{enumerate}
\end{theorem}
%
\begin{proof}
Using the strong convexity of the objective in \eqref{subproblem2} (note that $\xi^k = 0$ as $P_2 = 0$) and the fact that $x^{k+1}$ minimizes this objective over $C$, we obtain that for any $x\in C$,
\begin{align}\label{stongconve}
&P_1(x^{k+1}) + \theta_k \max_{i = 1,\cdots,m}\left[\ling{i}{x^{k+1}}{y^k}\right]_+ + \frac{\theta_k L_g}{2}\| x^{k+1} - y^k\|^2 \notag\\
&\leq P_1(x) + \theta_k \max_{i = 1,\cdots,m}\left[\ling{i}{x}{y^k}\right]_+ + \frac{\theta_k L_g}{2}\| x - y^k \|^2 - \frac{\theta_k L_g}{2}\| x - x^{k+1}\|^2.
\end{align}
Now we are ready to prove the three items one by one.

(i): For any $k \ge 1$, we see that
\begin{align*}
&\frac{1}{\theta_{k+1}}\!\left(P_1(x^{k+1}) \!-\! \hat{m}\right) \!+\!\! \max_{i = 1,\cdots,m}\left[g_i(x^{k+1})\right]_+\overset{\rm(a)}{\leq} \frac{1}{\theta_k}\!\left(P_1(x^{k+1}) \!-\! \hat{m}\right)\! +\!\! \max_{i = 1,\cdots,m}\left[g_i(x^{k+1})\right]_+\\
&\overset{\rm(b)}{\leq} \frac{P_1(x^{k+1}) - \hat{m}}{\theta_k} + \max_{i = 1,\cdots,m}\left[\ling{i}{x^{k+1}}{y^k} + \frac{L_{g_i}}{2}\| x^{k+1} - y^k \|^2\right]_+\\
&\overset{\rm(c)}{\leq} \frac{P_1(x^{k+1}) - \hat{m}}{\theta_k} + \max_{i = 1,\cdots,m}\left[\ling{i}{x^{k+1}}{y^k} \right]_+ + \frac{L_{g}}{2}\| x^{k+1} - y^k \|^2\\
&\overset{\rm(d)}{\leq}\frac{P_1(x^{k}) - \hat{m}}{\theta_k} + \max_{i = 1,\cdots,m}\left[\ling{i}{x^k}{y^k}\right]_+ + \frac{L_{g}}{2}\| x^{k} - y^k \|^2 - \frac{L_g}{2}\| x^{k+1} - x^k \|^2\\
&\overset{\rm(e)}{\leq}\frac{P_1(x^{k}) - \hat{m}}{\theta_k} +  \max_{i = 1,\cdots,m}\left[g_i(x^k)\right]_+ + \frac{L_{g}}{2}\| x^{k} - y^k \|^2 - \frac{L_g}{2}\| x^{k+1} - x^k \|^2\\
&= \frac{1}{\theta_k}\!\left(\!P_1(x^{k}) - \hat{m} + \theta_k\max_{i = 1,\cdots,m}\left[g_i(x^k)\right]_+\!\right)\! +\! \frac{\beta_k^2 L_g}{2}\|x^k - x^{k-1}\|^2 - \frac{L_g}{2}\| x^{k+1} - x^k \|^2\\
&= E(x^{k},x^{k-1},\theta_k) - \frac{(1 - \beta_k^2)L_g}{2}\|x^k - x^{k-1}\|^2 - \frac{L_g}{2}\| x^{k+1} - x^k \|^2,
\end{align*}
where (a) holds thanks to $\theta_k \leq\theta_{k+1}$ and $\hat{m} = \inf\{P_1(x): x\in C\}$, (b) holds because of the Lipschitz continuity of $\nabla g_i$, (c) follows from $L_g = \max\{L_{g_i}:\; i=1,\dots,m\}$, (d) holds upon invoking \eqref{stongconve} with $x=x^k$ (as $x^k\in C$), (e) follows from the convexity of $g_i$, and the last equality follows from the definition of $E(x^{k},x^{k-1},\theta_{k})$. The desired inequality now follows immediately from the above display and the definition of $E(x^{k+1},x^{k},\theta_{k+1})$.

(ii): Using similar arguments as Lemma~\ref{lemma1} (but using item (i) in place of Remark~\ref{rebarh}, and noting that Assumption~\ref{A1} holds according to Remark~\ref{rem43}), one can show that (ii) holds. We omit its proof for brevity.

(iii): Let $\Lambda$ be the set of accumulation points of $\{x^k\}$ for notational simplicity. From Remark~\ref{rem43}, Theorem~\ref{subconver} and the formula for the subdifferential of $\max_{i=1,\ldots,m}[g_i(\cdot)]_+$ (see \cite[Exercise~8.31]{rock97a}),
we deduce that
\begin{equation}\label{set}
\emptyset\not=\Lambda\subseteq \Argmin F_{\hat{\theta}} =: S.
\end{equation}
Now, write $E_{\theta}(x, y) := E(x, y, \theta)$ for notational simplicity. By the definitions of $F_\eta$ in \eqref{definfalpha} and $E(x, y, \theta)$ in item (i), we see that $E_{\hat{\theta}}(x, y) = F_{\hat{\theta}}(x) + \frac{L_g}{2}\|x - y\|^2$, where $\hat\theta$ is as in \eqref{defH}. From Remark~\ref{rem43}, Theorem~\ref{alpha} and item (i), we have that for any $k\geq N_0$, it holds that $\theta_k = \hat\theta$ and
\begin{equation}\label{upperK}
E_{\hat{\theta}}(x^{k+1}, x^k)\leq E_{\hat{\theta}}(x^k, x^{k-1}) -\frac{L_g(1 - \bar{\beta}^2)}{2}\|x^k - x^{k-1}\|^2,
\end{equation}
where $\bar \beta = \sup_k \beta_k < 1$ (recall that $\ell_g = 0$ under Assumption~\ref{B1}).

Let $\widetilde{S} = \{(x^*, x^*): x^*\in S\}$ and $\widetilde{\Lambda} = \{(x^*, x^*): x^*\in \Lambda\}$.
In view of \eqref{set}, we have $F_{\hat\theta}(\bar x) = \inf F_{\hat\theta}$ for any $\bar x\in S$. Using this together with item (ii) and the definition of $E_{\hat\theta}$, one can show readily that whenever $\bar{x}\in S$
\begin{equation}\label{valueomega}
\bar{\omega} = E_{\hat{\theta}}(\bar{x}, \bar{x}) = F_{\hat{\theta}}(\bar{x}) = \inf_x F_{\hat{\theta}}(x) = \inf_{x,y} E_{\hat{\theta}}(x, y).
\end{equation}
Moreover, in view of \eqref{set} and the definition of $E_{\hat\theta}$, we have
\begin{equation}\label{set2}
\emptyset\not=\widetilde{\Lambda}\subseteq \widetilde{S}= \Argmin\limits_{x,y} E_{\hat{\theta}}(x, y).
\end{equation}
Furthermore, since $F_{\hat{\theta}}$ is a KL function with exponent $\frac{1}{2}$, we conclude from \cite[Theorem~3.6]{li18} that $E_{\hat{\theta}}$ is a KL function with exponent $\frac{1}{2}$.
Using this together with \eqref{valueomega}, \eqref{set2} and Lemma~\ref{KLinequ}, we deduce that there exist $\epsilon_0 >0$, $r_0>0$, and $c_0>0$ such that
\begin{equation}\label{erro}
\d((x, y), \widetilde{S})^2\leq c_0(E_{\hat{\theta}}(x, y) - \bar\omega),
\end{equation}
for any $(x, y)\in\dom \partial E_{\hat{\theta}}$ satisfying $\d((x, y), \widetilde{S})\leq \epsilon_0$ and $\bar\omega\leq E_{\hat{\theta}}(x, y) < \bar\omega + r_0$.

Next, notice that $\{(x^k, x^{k-1})\}\subseteq C\times C\subset \dom \partial E_{\hat{\theta}} = C\times \R^n$, and we have from Theorem~\ref{suffdec}(iii) that $\widetilde{\Lambda}$ is the set of accumulation points of the bounded sequence $\{(x^k,x^{k-1})\}$. Using this and \eqref{set2}, we deduce that there exists $k_1\in \mathbb{N}$ such that
\begin{equation}\label{erro1}
\d((x^k, x^{k-1}), \widetilde{S})\leq\d((x^k, x^{k-1}), \widetilde{\Lambda})\leq\epsilon_0~~~~ \forall k\geq k_1.
\end{equation}
On the other hand, from Remark~\ref{rem43}, Theorem~\ref{alpha} and item (ii), we deduce the existence of $k_2\in \mathbb{N}$ such that
\begin{equation}\label{Ferro1}
\bar\omega\leq E_{\hat{\theta}}(x^k, x^{k-1})<\bar\omega + r_0 ~~~~ \forall k\geq k_2.
\end{equation}
Combining \eqref{erro}, \eqref{erro1} and \eqref{Ferro1}, we conclude that for any $k\geq k_3:=\max\{k_1, k_2\}$,
\begin{equation}\label{Eerro}
\d(x^{k}, S)^2\leq\d((x^k, x^{k-1}), \widetilde{S})^2\leq c_0(E_{\hat{\theta}}(x^k, x^{k-1}) - \bar\omega).
\end{equation}

Next, let $\bar{x}^k\in S$ satisfy $\|x^k - \bar{x}^k\| = \d(x^k, S)$. Then for any $k\geq N_0$ (note that $N_0$ is defined in Theorem~\ref{alpha}) and $\gamma\in(\frac{L_g c_0}{1 + L_gc_0}, 1)$, we have
\begin{align}\label{upperF}
&F_{\hat{\theta}}(x^{k+1})= \frac{1}{\hat{\theta}}\left(P_1(x^{k+1}) - \hat{m} + \hat{\theta}\max_{i = 1,\cdots,m}\left[g_i(x^{k+1})\right]_+\right) \nonumber\\
&\overset{\rm (a)}\leq \frac{1}{\hat{\theta}}\left(P_1(x^{k+1}) - \hat{m} + \hat{\theta}\max_{i = 1,\cdots,m}\left[\ling{i}{x^{k+1}}{y^k} + \frac{L_{g_i}}{2}\|x^{k+1} - y^k\|^2\right]_+\right) \nonumber\\
&\overset{\rm (b)}\leq\frac{1}{\hat{\theta}}\left(P_1(x^{k+1}) - \hat{m} + \hat{\theta}\max_{i = 1,\cdots,m}\left[\ling{i}{x^{k+1}}{y^k}\right]_+ + \frac{\hat{\theta}L_{g}}{2}\|x^{k+1} - y^k\|^2\right)\nonumber\\
&\overset{\rm (c)}\leq\frac{1}{\hat{\theta}}\left(P_1(\bar{x}^k) - \hat{m}\right) + \max_{i = 1,\cdots,m}\left[\ling{i}{\bar{x}^k}{y^k}\right]_+ + \frac{L_g}{2}\|\bar{x}^k - y^k\|^2 - \frac{L_g}{2}\|\bar{x}^k - x^{k+1}\|^2 \nonumber\\
&\overset{\rm (d)}\leq F_{\hat{\theta}}(\bar{x}^k) + \frac{L_g}{2}\|\bar{x}^k - y^k\|^2 - \frac{L_g}{2}\|\bar{x}^k - x^{k+1}\|^2, \nonumber\\
&\overset{\rm (e)}\leq F_{\hat{\theta}}(\bar{x}^k) + \frac{L_g}{2}\left(\|\bar{x}^k - x^k\| + \|x^k - y^k\|\right)^2 - \frac{L_g}{2}\|\bar{x}^k - x^{k+1}\|^2, \nonumber\\
&\overset{\rm(f)}\leq F_{\hat{\theta}}(\bar{x}^k)  + \frac{L_g}{2\gamma}\|\bar{x}^k - x^k\|^2 + \frac{L_g}{2(1 - \gamma)}\|x^k - y^k\|^2 - \frac{L_g}{2}\|\bar{x}^k - x^{k+1}\|^2\nonumber\\
&\overset{\rm(g)}\leq \bar\omega + \frac{L_g}{2\gamma}\d(x^k, S)^2 + \frac{L_g}{2(1 - \gamma)}\|x^k - y^k\|^2 - \frac{L_g}{2}\d(x^{k+1}, S)^2,
\end{align}
where (a) holds because of the Lipschitz continuity of $\nabla g_i$, (b) holds because $L_g = \max_{i = 1,\cdots,m}\{L_{g_i}\}$, (c) follows from \eqref{stongconve} with $x = \bar{x}^k$ (thanks to $\bar{x}^k\in S\subseteq C$ and the fact that $\theta_k = \hat\theta$ for $k\ge N_0$), (d) follows from the convexity of $g_i$ so that $\ling{i}{\bar{x}^k}{y^k}\le g_i(\bar{x}^k)$ for all $i$, (e) follows from the triangle inequality, (f) follows from the fact that $(a + b)^2 = (\gamma\frac{a}{\gamma} + (1 - \gamma)\frac{b}{(1-\gamma)})^2\leq\frac{a^2}{\gamma} + \frac{b^2}{(1-\gamma)}$ as $\gamma\in(0,1)$, and (g) holds thanks to \eqref{valueomega} and the definition of $\bar{x}^k$.

Then, we have for any $k\geq k_4:= \max\{k_3, N_0\}$ that
\begin{align*}
&E_{\hat{\theta}}(x^{k+1},x^k) - \bar\omega
= F_{\hat{\theta}}(x^{k+1}) - \bar\omega + \frac{L_g}{2}\|x^{k+1} - x^k\|^2 \\
&\overset{\rm(a)}\leq \frac{L_g}{2\gamma}\d(x^k, S)^2 + \frac{L_g}{2(1 - \gamma)}\|x^k - y^k\|^2 - \frac{L_g}{2\gamma}\d(x^{k+1}, S)^2  \nonumber\\
&~~~~+ \frac{L_g}{2}\left(\frac{1}{\gamma} - 1\right)\d(x^{k+1}, S)^2 + \frac{L_g}{2}\|x^{k+1} - x^k\|^2 \\
&\overset{\rm(b)}\leq \left(\frac{L_g}{2\gamma}\d(x^k, S)^2 - \frac{L_g}{2}\|x^k - x^{k-1}\|^2\right) - \left(\frac{L_g}{2\gamma}\d(x^{k+1}, S)^2 - \frac{L_g}{2}\|x^{k+1} - x^k\|^2\right) \nonumber\\
&~~~~+ \frac{L_g\bar{\beta}^2}{2(1 - \gamma)}\|x^k - x^{k-1}\|^2 + \frac{L_g}{2}\|x^k - x^{k-1}\|^2 + \frac{L_g}{2}\left(\frac{1}{\gamma} - 1\right)c_0(E_{\hat{\theta}}(x^{k+1}, x^{k}) - \bar\omega),
\end{align*}
where (a) holds because of \eqref{upperF}, and (b) follows from \eqref{Eerro}, $y^k = x^k + \beta_k(x^k - x^{k-1})$ and $\bar{\beta} = \sup_k\beta_k$.

Now, notice that $\gamma\in(\frac{L_g c_0}{1 + L_gc_0}, 1)$ implies $\frac{L_g}{2}(\frac{1}{\gamma} - 1)c_0 < \frac{1}{2}$. Letting $\vartheta: = 1 - \frac{L_g}{2}(\frac{1}{\gamma} - 1)c_0$, then we known that $\vartheta > \frac{1}{2}$. Rearranging terms in the above display inequality, we have that for any $k\ge k_4$,
\begin{align*}\label{upperE1}
&\vartheta \left(E_{\hat{\theta}}(x^{k+1},x^k) - \bar\omega\right) \\
& \leq \frac{L_g}{2}\left(\frac{1}{\gamma}\d(x^k, S)^2 - \|x^k - x^{k-1}\|^2\right) - \frac{L_g}{2}\left(\frac{1}{\gamma}\d(x^{k+1}, S)^2 -\|x^{k+1} - x^k\|^2\right) \nonumber\\
&~~~~+ \frac{L_g(1 - \gamma + \bar{\beta}^2)}{2(1 - \gamma)}\|x^k - x^{k-1}\|^2 \nonumber\\
&\overset{\rm(a)}\leq \frac{L_g}{2}\left(\frac{1}{\gamma}\d(x^k, S)^2 - \|x^k - x^{k-1}\|^2\right) - \frac{L_g}{2}\left(\frac{1}{\gamma}\d(x^{k+1}, S)^2 -\|x^{k+1} - x^k\|^2\right) \nonumber\\
&~~~~+ \frac{L_g(1 - \gamma + \bar{\beta}^2)}{2(1 - \gamma)}\cdot\frac{2}{L_g(1 - \bar{\beta}^2)}\left(E_{\hat{\theta}}(x^k, x^{k-1}) - E_{\hat{\theta}}(x^{k+1}, x^k)\right), \nonumber
\end{align*}
where (a) follows from \eqref{upperK}.

Denote $\zeta: = \frac{1 + \bar{\beta}^2 - \gamma}{(1-\gamma)(1 - \bar{\beta}^2)} > 1$ and $A_k := \frac{L_g}{2}(\frac{1}{\gamma}\d(x^k, S)^2 - \|x^k - x^{k-1}\|^2)$. Rearranging terms in the above inequality, we obtain that for any $k\ge k_4$,
\begin{equation*}
(\vartheta + \zeta) \left(E_{\hat{\theta}}(x^{k+1},x^k) - \bar\omega\right)  \leq A_k - A_{k+1} + \zeta\left(E_{\hat{\theta}}(x^{k},x^{k-1}) - \bar\omega\right).
\end{equation*}
Dividing $\vartheta + \zeta$ on both sides in the above inequality, we see that for any $k\ge k_4$,
\begin{equation}\label{upperE2}
E_{\hat{\theta}}(x^{k+1},x^k) - \bar\omega \leq \frac{\zeta}{\vartheta + \zeta}\left(E_{\hat{\theta}}(x^{k},x^{k-1}) - \bar\omega\right) + \frac{1}{\vartheta + \zeta}A_k - \frac{1}{\vartheta + \zeta}A_{k+1}.
\end{equation}
Since $\vartheta > \frac{1}{2}$ and $\zeta > 1$, we have that for any $k\ge k_4$,
\begin{align}\label{upperA}
&\left|\frac{A_k}{\vartheta + \zeta}\right|\leq \left|A_k\right|\leq \frac{L_g}{2}\left(\frac{1}{\gamma}\d(x^k, S)^2 + \|x^k - x^{k-1}\|^2\right) \nonumber\\
&\overset{\rm(a)}\leq\frac{L_g c_0}{2\gamma}\left(E_{\hat{\theta}}(x^k, x^{k-1}) - \bar\omega\right) + \frac{L_g}{2}\|x^k - x^{k-1}\|^2\nonumber\\
&\overset{\rm(b)}\leq\frac{L_g c_0}{2\gamma}\left(E_{\hat{\theta}}(x^k, x^{k-1}) - \bar\omega\right) + \frac{1}{(1 - \bar{\beta}^2)}\left(E_{\hat{\theta}}(x^k, x^{k-1}) - E_{\hat{\theta}}(x^{k+1}, x^{k})\right)\nonumber\\
&\overset{\rm(c)} \leq c_1\left(E_{\hat{\theta}}(x^k, x^{k-1}) - \bar\omega\right),
\end{align}
where (a) holds thanks to \eqref{Eerro}, (b) holds because of \eqref{upperK}, and (c) follows from $E_{\hat{\theta}}(x^{k+1}, x^{k})\geq\bar\omega$ (see \eqref{valueomega}) with $c_1:= \frac{L_g c_0}{2\gamma} + \frac{1}{(1 - \bar{\beta}^2)}$.

Let $\varrho = \frac{c_1 + \frac{\zeta}{\vartheta+\zeta}}{c_1+1}\in(0,1)$. Then one can see that
\begin{equation}\label{divi}
\frac{\zeta}{\vartheta+\zeta} + (1 - \varrho)c_1 = \varrho.
\end{equation}
Then, from \eqref{upperE2}, we obtain that for any $k\ge k_4$,
\begin{align*}
&E_{\hat{\theta}}(x^{k+1},x^k) - \bar\omega + \frac{1}{\vartheta + \zeta}A_{k+1} \leq \frac{\zeta}{\vartheta + \zeta}\left(E_{\hat{\theta}}(x^{k},x^{k-1}) - \bar\omega\right) + \frac{1}{\vartheta + \zeta}A_k \nonumber\\
&\overset{\rm(a)} \leq \frac{\zeta}{\vartheta + \zeta}\left(E_{\hat{\theta}}(x^{k},x^{k-1}) - \bar\omega\right) + \frac{\varrho}{\vartheta + \zeta}A_k + (1 - \varrho)\left|\frac{A_k}{\vartheta + \zeta}\right|\nonumber\\
&\overset{\rm(b)} \leq \left(\frac{\zeta}{\vartheta + \zeta} + (1 - \varrho)c_1\right)\left(E_{\hat{\theta}}(x^{k},x^{k-1}) - \bar\omega\right) + \frac{\varrho}{\vartheta + \zeta}A_k\nonumber\\
&\overset{\rm(c)} = \varrho \left(E_{\hat{\theta}}(x^{k},x^{k-1}) - \bar\omega + \frac{1}{\vartheta + \zeta}A_k \right),
\end{align*}
where (a) holds as $\varrho\in(0, 1)$, (b) follows from \eqref{upperA}, and (c) holds because of \eqref{divi}.

Inductively, since $\varrho > 0$, we see that for any $k\geq k_4$,
\begin{equation*}
E_{\hat{\theta}}(x^{k+1},x^k) - \bar\omega + \frac{1}{\vartheta + \zeta}A_{k+1} \leq \varrho^{k - k_4 +1}\left(E_{\hat{\theta}}(x^{k_4},x^{k_4 - 1}) - \bar\omega + \frac{1}{\vartheta + \zeta}A_{k_4}\right),
\end{equation*}
which means, there exists $M>0$ such that, for any $k\geq k_4$,
\begin{align}\label{inequ}
0&\leq E_{\hat{\theta}}(x^k,x^{k-1}) - \bar\omega\leq M\varrho^{k} - \frac{1}{\vartheta + \zeta}A_{k} \notag\\
&\overset{\rm(a)}= M\varrho^{k} \!-\! \frac{L_g}{2(\vartheta + \zeta)}\left(\!\frac{1}{\gamma}\d(x^k, S)^2 - \|x^k - x^{k-1}\|^2\!\right)\!\leq\! M\varrho^{k} \!+\! \frac{L_g}{2(\vartheta + \zeta)} \|x^k - x^{k-1}\|^2\notag\\
&\overset{\rm(b)}\leq M\varrho^{k} + \frac{1}{(\vartheta + \zeta)(1 - \bar{\beta}^2)}\left(E_{\hat{\theta}}(x^k,x^{k-1}) - E_{\hat{\theta}}(x^{k+1},x^{k})\right),
\end{align}
where (a) follows from the definition of $A_k$, and (b) holds because of \eqref{upperK}.

Taking $\mu > \max\{\frac{1}{(\vartheta + \zeta)(1 - \bar{\beta}^2)}, \frac{1}{1 - \varrho}\}$. From $\varrho\in(0, 1)$,  we see that
\begin{equation}\label{defmu}
\mu > 1 \text{ and } 1 - \mu^{-1} > \varrho,
\end{equation}
and from \eqref{inequ} (and \eqref{upperK}, which asserts the nonnegativity of the difference $E_{\hat{\theta}}(x^k,x^{k-1}) - E_{\hat{\theta}}(x^{k+1},x^{k})$), we have that
\begin{equation*}
E_{\hat{\theta}}(x^k,x^{k-1}) - \bar\omega \leq M\varrho^{k} + \mu(E_{\hat{\theta}}(x^k,x^{k-1}) - E_{\hat{\theta}}(x^{k+1},x^{k})),
\end{equation*}
which implies
\begin{equation*}
\mu(E_{\hat{\theta}}(x^{k+1},x^{k}) - \bar\omega) \leq (\mu - 1)\left(E_{\hat{\theta}}(x^k,x^{k-1}) - \bar\omega\right) + M\varrho^{k}.
\end{equation*}
Dividing $\mu>0$ on the both sides in the above display, we see that for any $k\geq k_4$,
\begin{align*}
&E_{\hat{\theta}}(x^{k+1},x^{k}) - \bar\omega \leq \left(1 - \mu^{-1}\right) \left(E_{\hat{\theta}}(x^k,x^{k-1}) - \bar\omega\right) + \frac{M}{\mu}\varrho^{k} \nonumber\\
&= \left(1 - \mu^{-1}\right) \left(E_{\hat{\theta}}(x^k,x^{k-1}) - \bar\omega\right) + \frac{M}{\mu}\left( \frac{1 - \mu^{-1}}{1 - \mu^{-1} - \varrho} - \frac{\varrho}{1 - \mu^{-1} - \varrho}\right)\varrho^{k} \nonumber\\
&= \left(1 - \mu^{-1}\right) \left(E_{\hat{\theta}}(x^k,x^{k-1}) - \bar\omega + \frac{M}{\mu(1 - \mu^{-1} - \varrho)}\varrho^{k} \right) - \frac{M}{\mu(1 - \mu^{-1} - \varrho)}\varrho^{k+1}, \nonumber
\end{align*}
where the division by $1 - \mu^{-1} - \varrho$ is valid thanks to \eqref{defmu}.
Rearranging terms in the above display inequality, we have that for any $k\geq k_4$,
\begin{equation*}
E_{\hat{\theta}}(x^{k+1},x^{k}) - \bar\omega + \frac{M\varrho^{k+1}}{\mu(1 - \mu^{-1} - \varrho)}
\leq \!\left(1 - \mu^{-1}\right)\! \left(E_{\hat{\theta}}(x^k,x^{k-1}) - \bar\omega + \frac{M\varrho^{k}}{\mu(1 - \mu^{-1} - \varrho)} \right)\!.
\end{equation*}
Inductively, since $1 - \mu^{-1} >0$ thanks to \eqref{defmu}, we see that for any $k\geq k_4$,
\begin{align}\label{inductE}
&\!\!\!E_{\hat{\theta}}(x^{k+1},x^{k}) - \bar\omega
 \leq
 E_{\hat{\theta}}(x^{k+1},x^{k}) - \bar\omega + \frac{M\varrho^{k+1}}{\mu(1 - \mu^{-1} - \varrho)} \nonumber\\
&\!\!\!\leq \left(\!1 - \mu^{-1}\!\right)^{k-k_4+1} \!\!\!\ \left(E_{\hat{\theta}}(x^{k_4},x^{k_4-1}) \!-\! \bar\omega \!+\! \frac{M\varrho^{k_4}}{\mu(1 - \mu^{-1} - \varrho)} \right) \overset{\rm(a)} = c_2 \left(\!1 - \mu^{-1}\!\right)^{k + 1},
\end{align}
where (a) holds with $c_2 := \left(1 - \mu^{-1}\right)^{-k_4}\left(E_{\hat{\theta}}(x^{k_4},x^{k_4-1}) - \bar\omega + \frac{M}{\mu(1 - \mu^{-1} - \varrho)}\varrho^{k_4} \right)> 0$.

Finally, we obtain that for any $k\geq k_4$,
\begin{align*}
\|x^{k+1} - x^k\|^2 & \overset{\rm(a)}\leq \frac{2}{L_g(1 - \bar{\beta}^2)}\left(E_{\hat{\theta}}(x^{k+1}, x^k) - E_{\hat{\theta}}(x^{k+2}, x^{k+1})\right)\\
&\overset{\rm(b)}\leq\frac{2}{L_g(1 - \bar{\beta}^2)} \left(E_{\hat{\theta}}(x^{k+1}, x^{k}) - \bar\omega\right)\overset{\rm(c)}\leq \frac{2c_2}{L_g(1 - \bar{\beta}^2)}\left(1 - \mu^{-1}\right)^{k+1},
\end{align*}
where (a) holds because of \eqref{upperK}, (b) follows from $E_{\hat{\theta}}(x^{k+1}, x^k) \geq \bar\omega$ (see \eqref{valueomega}), and (c) holds because of \eqref{inductE}.
Consequently, for any $j\ge k\geq k_4$,
\begin{align}\label{bound}
\sum_{i = k}^{j}\|x^{i+1} - x^i\|&\leq \sum_{i = k}^{\infty} \sqrt{\frac{2c_2}{L_g(1 - \bar{\beta}^2)}}\left(\sqrt{1 - \mu^{-1}}\right)^{i+1}= c_3 \left(\sqrt{1 - \mu^{-1}}\right)^{k+1},
\end{align}
with $c_3 := \sqrt{\frac{2c_2}{L_g(1 - \bar{\beta}^2)}}\cdot\frac{1}{1 - \sqrt{1 - \mu^{-1}}} > 0$,
which implies that $\{x^k\}$ is a Cauchy sequence. Combining this with Remark~\ref{rem43} and Theorem~\ref{subconver}, we see that $\{x^k\}$ converges to a minimizer $x^*$ of \eqref{eq1}. The claimed linear rate of convergence also follows immediately from \eqref{bound}.
\end{proof}

\section{KL exponent and exact penalty}\label{sec5}

In Section~\ref{sec42}, the KL exponent of the $F_{\hat\theta}$ in Theorem~\ref{Edecres}(iii) was used for establishing the convergence rate of the $\{x^k\}$ generated by Algorithm~\ref{alg:Framwork} in the convex setting. In this section, we examine how the KL exponent of functions of the form \eqref{definfalpha} can be deduced from the corresponding problem \eqref{eq1}.

Specifically, we consider the following optimization problem:
\begin{equation}\label{KLproblem}
\min_{x\in \R^n}\hat F(x):= P_1(x) + \delta_C(x) + \delta_{{\cal F}}(x),
\end{equation}
where $P_1:\R^n \to \R$ is convex, $C$ is compact and convex, ${\cal F} := \{x\in \R^n: g_i(x)\le 0, i = 1,\ldots,m\}$ with each $g_i:\mathbb{R}^n\to \R$ being convex, and $C\cap \{x\in \R^n:\; \max_{i=1,\ldots,m}\{g_i(x)\} < 0\}\neq \emptyset$; we also consider the associated penalty function
\begin{equation}\label{Feta}
\hat F_{\eta}(x): = P_1(x) + \delta_C(x) + \eta\max_{i=1,\ldots,m}[g_i(x)]_+,
\end{equation}
where $\eta > 0$. Notice that for \eqref{eq1}, the KL property of the corresponding $\hat F_{\hat\theta}$ was the key for establishing the convergence rate of the $\{x^k\}$ generated by Algorithm~\ref{alg:Framwork}; see Theorem~\ref{Edecres}(iii).

We next recall the definition of exact penalty parameter.
\begin{definition}[Exact penalty parameter]\label{Exact}
Consider \eqref{KLproblem} and \eqref{Feta}. If there exists $\bar{\eta} > 0$ such that for all $\eta \geq \bar{\eta}$,
\[
\Argmin_{x\in \R^n}\hat F_\eta(x) = \Argmin_{x\in \R^n}\hat F(x),
\]
then $\bar{\eta}$ is called an exact penalty parameter of \eqref{KLproblem}.
\end{definition}
We will argue that the set of exact penalty parameters of \eqref{KLproblem} is nonvoid. We start by recalling the following well-known result, whose short proof is included for the convenience of the readers.
\begin{lemma}\label{distF}
  Let $C$ and ${\cal F}$ be as in \eqref{KLproblem}. Then there exist $\kappa>0$ and $\tau > 0$ such that
  \begin{equation}\label{eb2}
  \d(x,C\cap {\cal F})\le \kappa \d(x,{\cal F}) \le \tau \max_{i=1,\ldots,m}[g_i(x)]_+\ \ \ \forall x\in C.
  \end{equation}
\end{lemma}
\begin{proof}
  First, since $C\cap \{x\in \R^n:\; \max_{i=1,\ldots,m}\{g_i(x)\} < 0\}\neq \emptyset$ (say, it contains $\hat x$), we deduce from \cite[Corollary~3]{BauschkeBorweinLi99} that there exists $\kappa > 0$ such that
  \begin{equation}\label{eb1}
  \d(x,C\cap {\cal F}) \le \kappa \d(x,{\cal F})\ \ \ \ \ \forall x\in C.
  \end{equation}
  We then apply Lemma~\ref{RobEB} with $\Omega:=\mathcal{F}$, $h(x) = (g_1(x), g_2(x), \ldots, g_m(x))$, $x^s = \hat{x}$ and $\delta_0 = \left|\max_{i=1,\ldots,m}\{g_i(\hat{x})\}\right|$ to obtain
\begin{equation*}
\d(x, \mathcal{F}) \leq \frac{\|x - \hat{x}\|}{\left|\max_{i=1,\ldots,m}\{g_i(\hat{x})\}\right|}\d(0, g(x) + \R^m_+) ~~~~ \forall x\in \R^n.
\end{equation*}
Since $C$ is compact, we deduce further that there exists $M_1> 0 $ such that
\begin{equation}\label{upperxF}
\d(x, \mathcal{F}) \leq M_1\max_{i = 1,\cdots,m}[g_i(x)]_+ ~~ \forall x\in C.
\end{equation}
The desired conclusion now follows upon combining \eqref{eb1} and \eqref{upperxF}.
\end{proof}
\begin{remark}[Nonemptiness of the set of exact penalty parameters]\label{rem51}
Consider \eqref{KLproblem} and \eqref{Feta}.
Since $C\cap {\cal F}$ is compact and $P_1$ is continuous, we see that $\Argmin \hat F\neq \emptyset$. Using this together with \eqref{eb2}, we can now deduce from \cite[Lemma~3.1]{ChenLuPong16} that any $\eta > \bar L_{P_1}\tau$ is an exact penalty parameter of \eqref{KLproblem}, where $\bar L_{P_1}$ is a Lipschitz continuity modulus for $P_1$ on the compact convex set $C$.
\end{remark}

Now, we show that if the $\hat F$ in \eqref{KLproblem} is a KL function with exponent $\alpha\in (0,1)$, then for any $\eta>\bar{\eta}$, the $\hat F_{\eta}$ in \eqref{Feta} is a KL function with the same exponent, where $\bar{\eta}$ is an exact penalty parameter of \eqref{KLproblem}.

\begin{theorem}[KL exponent of $\hat F_{\eta}$ from that of $\hat F$]\label{FetaKL}
Let $\hat F$ be as in \eqref{KLproblem}, $\bar{x}\in\Argmin \hat F$ and $\bar{\eta}$ be an exact penalty parameter of \eqref{KLproblem}. If $\hat F$ satisfies the KL property with exponent $\alpha\in (0,1)$ at $\bar{x}$, then for any $\eta>\bar{\eta}$, the $\hat F_{\eta}$ defined in \eqref{Feta} satisfies the KL property with exponent $\alpha$ at $\bar{x}$.
\end{theorem}

\begin{proof}
Fix any $\eta > \bar\eta$. Since $\bar{\eta}$ is an exact penalty parameter of \eqref{KLproblem}, we see that $\Argmin \hat F = \Argmin \hat F_{\eta}$; also, note that $\dom\partial \hat F_{\eta} = C$ and $\dom\partial \hat F = C\cap {\cal F}$.

Since $\hat F$ satisfies the KL property with exponent $\alpha$ at $\bar{x}$, in view of \cite[Theorem~5]{bolte17}, there exist $c>0$ and $a$, $\epsilon\in (0,1)$ such that
\begin{equation}\label{FKL}
\d(x, \Argmin \hat F) \leq c(\hat F(x) - \hat F(\bar{x}))^{1-\alpha},
\end{equation}
whenever $ x\in\dom \partial \hat F = C\cap \mathcal{F}$ satisfies $\|x - \bar{x}\|\leq \epsilon$ 
and $\hat F(\bar{x}) \leq \hat F(x) \leq \hat F(\bar{x}) + a$. Since $\hat F$ is continuous on its domain, by shrinking $\epsilon$ further if necessary, we assume that \eqref{FKL} holds whenever $ x\in\dom \partial \hat F = C\cap \mathcal{F}$ satisfies $\|x - \bar{x}\|\leq \epsilon$.

Next, since $P_1:\R^n\to\R$ is convex, we know that $P_1$ is locally Lipschitz continuous at $\bar{x}$. Hence, there exist $\bar{\epsilon} > 0$ and $\hat L_{P_1}>0$ such that
\begin{equation}\label{LipschP}
|P_1(x) - P_1(y)|\leq \hat L_{P_1}\|x - y\| ~~~~\forall x,y\in \{u\in \R^n: \|u - \bar x\|\le \bar\epsilon\}.
\end{equation}

Now, take $\epsilon_0 := \min\{\epsilon, \bar{\epsilon}\}$. Then for any $x\in C = \dom\partial \hat F_{\eta}$ satisfying $\|x - \bar{x}\|\leq \epsilon_0$, we have upon letting $\Pi_{C\cap\mathcal{F}}(x)$ denote the orthogonal projection of $x$ onto $C\cap \mathcal{F}$ that
\begin{align}\label{distFeta}
&\d(x, \Argmin \hat F_{\eta})\leq \d(\Pi_{C\cap \mathcal{F}} (x), \Argmin \hat F_{\eta}) + \d(x, C\cap \mathcal{F}) \notag\\
&\overset{\rm(a)} = \d(\Pi_{C\cap \mathcal{F}} (x), \Argmin \hat F) + \d(x, C\cap \mathcal{F}) \notag\\
&\overset{\rm(b)}\leq \!c(\hat F(\Pi_{C\cap \mathcal{F}} (x)) \!-\! \hat F(\bar{x}))^{1-\alpha} \!\!+\! \kappa\d(x, \mathcal{F})\!=\! c(P_1(\Pi_{C\cap \mathcal{F}} (x)) \!-\! P_1(\bar{x}))^{1-\alpha} \!\!+\! \kappa\d(x, \mathcal{F}) \notag\\
&\overset{\rm(c)}\leq c(P_1(x) - P_1(\bar{x}) + \hat L_{P_1}\d(x, C\cap \mathcal{F}))^{1-\alpha} + \kappa\d(x, \mathcal{F}) \notag\\
&\overset{\rm(d)}\leq c(P_1(x) - P_1(\bar{x}) + \hat L_{P_1}\kappa\d(x, \mathcal{F}))^{1-\alpha} + \kappa\d(x, \mathcal{F})^{1-\alpha} \notag\\
& \overset{\rm (e)}\le\hat c \big[(P_1(x) - P_1(\bar{x}) + \hat L_{P_1}\kappa\d(x, \mathcal{F}))^{1-\alpha} + [\kappa^{1/({1-\alpha})}\d(x, \mathcal{F})]^{1-\alpha} \big]\notag\\
&\overset{\rm(f)}\leq \!\bar{c}(P_1(x) \!-\! P_1(\bar{x}) \!+\! \kappa_1\d(x, \mathcal{F}))^{1-\alpha}
\!\!\overset{\rm(g)}\leq\! \bar{c}\big(P_1(x) \!-\! P_1(\bar{x}) \!+\! \kappa_2\!\!\max_{i = 1,\cdots,m}\![g_i(x)]_+\big)^{1-\alpha},\!\!
\end{align}
where (a) holds because $\Argmin \hat F = \Argmin \hat F_{\eta}$, (b) holds because of \eqref{eb2}, \eqref{FKL} and the fact that $\|\Pi_{C\cap\mathcal{F}}(x) - \bar{x}\|\leq\|x - \bar{x}\|\leq\epsilon_0\le\epsilon$ (thanks to $\bar{x}\in C\cap\mathcal{F}$ and the projection mapping being nonexpansive), (c) follows from \eqref{LipschP} and the fact that $\epsilon_0 \le \bar\epsilon$, (d) follows from \eqref{eb2} and the facts that $\d(x,\mathcal{F})\leq \|x - \bar{x}\|\leq\epsilon_0\le \epsilon<1$ and $\alpha\in (0,1)$, (e) holds with $\hat{c} = \max\{c, 1\}$, (f) holds with $\bar c = 2^{\alpha}\hat c$ and $\kappa_1 = \hat L_{P_1}\kappa + \kappa^{\frac{1}{1-\alpha}}$ thanks to the fact that $a^{1-\alpha} + b^{1-\alpha}\leq 2^{\alpha}(a + b)^{1-\alpha}$ for any $a$, $b\ge0$ and $\alpha\in (0,1)$, and (g) holds with $\kappa_2 = \kappa_1\tau/\kappa$ thanks to \eqref{eb2}.

Now, if $\eta\geq\kappa_2$, then, from \eqref{distFeta}, we have that
\[
\d(x, \Argmin \hat F_{\eta})\!\leq\! \bar{c}\!\left(P_1(x) - P_1(\bar{x}) + \eta\max_{i = 1,\cdots,m}[g_i(x)]_+\right)^{1-\alpha}\!\!\!\!\!
=\bar{c}(\hat F_{\eta}(x) - \hat F_{\eta}(\bar{x}))^{1-\alpha}.
\]
On the other hand, if $\kappa_2 >\eta>\bar{\eta}$, then, from \eqref{distFeta}, we obtain that
\begin{align*}
&\d(x, \Argmin \hat F_{\eta})\!\leq \!\bar{c}\left(\!P_1(x) \!-\! P_1(\bar{x}) \!+\! \bar{\eta}\!\max_{i = 1,\cdots,m}[g_i(x)]_+ \!+\! (\kappa_2 - \bar{\eta})\!\!\max_{i = 1,\cdots,m}[g_i(x)]_+\!\!\right)^{1-\alpha}\\
&\overset{\rm(a)} \leq \bar{c}\left(\frac{\kappa_2 - \bar{\eta}}{\eta - \bar{\eta}}\right)^{1-\alpha}\!\!\left(P_1(x) - P_1(\bar{x}) + \bar{\eta}\max_{i = 1,\cdots,m}[g_i(x)]_+ + (\eta - \bar{\eta}) \max_{i = 1,\cdots,m}[g_i(x)]_+\right)^{1-\alpha}\\
&= \bar{c}\left(\frac{\kappa_2 - \bar{\eta}}{\eta - \bar{\eta}}\right)^{1-\alpha}\!\!\left(\hat F_{\eta}(x) - \hat F_{\eta}(\bar{x})\right)^{1-\alpha},
\end{align*}
where (a) holds because $a + b\leq\frac{1}{\epsilon}(a + \epsilon b)$ for any $a\geq 0$, $b \geq 0$ and $0<\epsilon\leq 1$.\footnote{We apply this relation to $\epsilon := (\eta-\bar\eta)/(\kappa_2 - \bar\eta)\in (0,1)$, $b:= (\kappa_2 - \bar{\eta})\max_{i = 1,\cdots,m} [g_i(x)]_+ \geq 0$, and $a:= P_1(x) - P_1(\bar{x}) + \bar{\eta}\max_{i = 1,\cdots,m}[g_i(x)]_+$, which is nonnegative because $\bar{\eta}$ is an exact penalty parameter, $\bar{x}\in\Argmin \hat F = \Argmin \hat F_{\bar\eta}$ and $x\in C$.} The desired conclusion now follows immediately upon invoking \cite[Theorem~5]{bolte17}.
\end{proof}

We next comment on how Theorem~\ref{FetaKL} can be applied to find the KL exponent of $F_{\hat\theta}$ in Theorem~\ref{Edecres}(iii); specifically, we will comment on the condition $\eta > \bar\eta$ in Theorem~\ref{FetaKL}. We first recall the following well-known result concerning exact penalty parameters.
\begin{lemma}\label{lem43}
  Consider \eqref{KLproblem} and \eqref{Feta}. If
  $\widetilde\eta > 0$ is such that $\Argmin \hat F_{\widetilde\eta}\cap \Argmin_{x\in{C\cap {\cal F}}}P_1(x)\neq \emptyset$, then $\Argmin \hat F_{\eta} = \Argmin_{x\in {C\cap {\cal F}}}P_1(x)$ whenever $\eta > \widetilde\eta$.
\end{lemma}
\begin{proof}
Fix any $\eta > \widetilde\eta$ and let $\hat x \in \Argmin \hat F_{\widetilde\eta}\cap \Argmin_{x\in{C\cap {\cal F}}}P_1(x)$.
We first argue that $\Argmin_{x\in{C\cap {\cal F}}}P_1(x) = \Argmin \hat F_{\eta}\cap {\cal F}$. Indeed, if $\tilde x\in \Argmin_{x\in {C\cap {\cal F}}}P_1(x)$, then $\tilde x \in {C\cap {\cal F}}\subseteq {\cal F}$ and hence $\max_{i=1,\ldots,m}[g_i(\tilde x)]_+ = 0$. Moreover, it holds that
\begin{align*}
\hat F_{\eta}(\tilde x) = P_1(\tilde x)\overset{\rm (a)}= P_1(\hat x) = \hat F_{\widetilde\eta}(\hat x)\overset{\rm (b)}\le  \hat F_{\widetilde\eta}(x)\overset{\rm (c)}\le \hat F_{\eta}(x)
\end{align*}
for any $x\in C$, where (a) holds because both $\hat x$ and $\tilde x$ minimize $P_1$ over ${C\cap {\cal F}}$, (b) holds because $\hat x$ also minimizes $\hat F_{\widetilde\eta}$, and (c) holds because $\eta> \widetilde\eta$. As for the converse inclusion, let $\tilde x \in \Argmin \hat F_{\eta}\cap {\cal F}$. Then for any $x\in C\cap {\cal F}$, we have
\[
P_1(\tilde x)= \hat F_{{\eta}}(\tilde x)\le \hat F_{{\eta}}(x) = P_1(x),
\]
where the equalities hold because $u\in {\cal F}$ implies $\max_{i=1,\ldots,m}[g_i(u)]_+=0$. The above arguments establish $\Argmin_{x\in{C\cap {\cal F}}}P_1(x) = \Argmin \hat F_{\eta}\cap {\cal F}$.

To complete the proof, it now suffices to show that $\Argmin \hat F_{\eta}\subseteq {\cal F}$. To this end, let $\tilde x \in \Argmin \hat F_{\eta}$. Then we have
\begin{align*}
&P_1(\tilde x) + \eta \max_{i=1,\ldots,m}[g_i(\tilde x)]_+=\hat F_{\eta}(\tilde x) \le \hat F_{\eta}(\hat x) = P_1(\hat x) + \eta \max_{i=1,\ldots,m}[g_i(\hat x)]_+ \\
&\overset{\rm (a)}= P_1(\hat x) + \widetilde\eta \max_{i=1,\ldots,m}[g_i(\hat x)]_+
= \hat F_{\widetilde\eta}(\hat x) \overset{\rm (b)}\le \hat F_{\widetilde\eta}(\tilde x) = P_1(\tilde x) + \widetilde\eta\max_{i=1,\ldots,m}[g_i(\tilde x)]_+,
\end{align*}
where (a) holds because $\hat x\in {C\cap {\cal F}}$ (hence $\max_{i=1,\ldots,m}[g_i(\hat x)]_+ = 0$) and (b) holds because $\hat x$ minimizes $\hat F_{\widetilde\eta}$. Rearranging terms in the above inequality, we obtain $(\eta-\widetilde\eta)\max_{i=1,\ldots,m}[g_i(\tilde x)]_+ = 0$, which means $\tilde x\in {\cal F}$.
\end{proof}
\begin{remark}[On the condition $\eta > \bar \eta$ in Theorem~\ref{FetaKL}]\label{rem52}
  We comment on the applicability of Theorem~\ref{FetaKL}, which only infers the KL exponent of $\hat F_\eta$ when $\eta > \bar \eta$ for some exact penalty parameter $\bar\eta$.

  Particularly, we consider \eqref{eq1}. Suppose that Assumptions~\ref{B1} and \ref{B2} hold and let $\{(x^k,\theta_k)\}$ be generated by Algorithm~\ref{alg:Framwork}. Using Theorem~\ref{subconver} (see also Remark~\ref{rem43}) and the formula for the subdifferential of $\max_{i=1,\ldots,m}[g_i(\cdot)]_+$ (see \cite[Exercise~8.31]{rock97a}), we deduce from the definition of $F_\eta$ in \eqref{definfalpha} that
\begin{equation}\label{conditionunhappy}
\emptyset\neq \Lambda\subseteq \Argmin F_{\hat\theta} \cap \Argmin_{x\in{C\cap \mathscr{F}}}P_1(x).
\end{equation}
where ${C\cap \mathscr{F}}$ is the feasible set of \eqref{eq1} and $\Lambda$ is the set of accumulation points of $\{x^k\}$.
Combining \eqref{conditionunhappy} with Lemma~\ref{lem43}, we deduce that the set of exact penalty parameters is nonempty; indeed, it contains the interval $(\hat\theta,\infty)$. Hence, if we let $\tilde \eta$ denote the infimum of the set of exact penalty parameters, then $\hat\theta\ge \tilde\eta$.

  Now, note that we have $\theta_k \equiv \hat\theta$ whenever $k\ge N_0$ (where $N_0$ is defined in Theorem~\ref{alpha}) and $\theta_k$ is nondecreasing. Intuitively, it is likely that the update rule of $\theta_k$ will result in $\hat\theta > \tilde\eta$. In this case, Theorem~\ref{FetaKL} asserts that the KL property required in Theorem~\ref{Edecres}(iii) can be inferred from that of $P_1+\delta_C+\delta_{\mathscr{F}}$ in \eqref{eq1}. On the other hand, in the case $\hat\theta = \tilde\eta$, Theorem~\ref{FetaKL} is not applicable for connecting the KL property of $F_{\hat\theta}$ to that of $P_1+\delta_C+\delta_{\mathscr{F}}$.
\end{remark}

\begin{example}\label{RemarkKL}
Suppose that in \eqref{KLproblem}, $P_1 = \|\cdot\|_1$, $C$ is a polytope containing the origin, $m = 1$, and $g_1 = q_1\circ A_1$ for some matrix $A_1\in \R^{s_1\times n}$ and $q_1:\mathbb{R}^{s_1} \rightarrow \mathbb{R}$ taking one of the following forms with $b\in \R^{s_1}$ and $\sigma > 0$ chosen so that the origin is not feasible and that $\inf_{x\in C}g_1(x) < 0$:
\begin{enumerate}[{\rm (i)}]
   \item (Basis pursuit denoising \cite{Ca18}) $q_1(z) = \frac{1}{2}\|z - b\|^2 - \sigma$.
   \item (Logistic loss \cite{HoLS13}) $q_1(z) = \sum_{i=1}^{s_1}\log(1 + \exp(b_iz_i)) - \sigma$ for some $b\in \{-1,1\}^{s_1}$.
 \end{enumerate}
 Let $\bar{\eta}$ be the exact penalty parameter of \eqref{KLproblem}. We deduce from \cite[Section~5.1]{zhang23} and Theorem~\ref{FetaKL} that, for any $\eta>\bar{\eta}$, the KL exponent of the corresponding $\hat F_{\eta}$ in \eqref{Feta} (and hence the corresponding $F_\eta$ in \eqref{definfalpha}) is $\frac{1}{2}$.
\end{example}

\section{Numerical experiments}\label{sec6}

In this section, we perform numerical experiments to illustrate the performance of Algorithm~\ref{alg:Framwork}. Particularly, we consider the following model for compressed sensing:
\begin{equation}\label{CompSen}
\begin{aligned}
\min_{x\in\R^n}\quad &\|x\|_1 - \mu\|x\| \\
{\rm s.t.} \quad & h(Ax - b)\leq \sigma,
\end{aligned}
\end{equation}
where $\mu\in[0, 1)$, $A\in\R^{q\times n}$ has full row rank, $b\in \R^q$, $h: \R^q\rightarrow\R_+$ is an analytic function whose gradient is Lipschitz continuous with modulus $L_{h}$ and satisfies $h(0) = 0$, and $\sigma\in (0, h(-b))$.

Although the feasible region of \eqref{CompSen} is unbounded and Algorithm~\ref{alg:Framwork} cannot be directly applied to solving \eqref{CompSen}, one can argue as in the discussion following \cite[Eq.~(6.2)]{zhang23} that \eqref{CompSen} is equivalent to the following model:
\begin{equation}\label{CompSen1}
\begin{aligned}
\min_{x\in\R^n}\quad &\|x\|_1 - \mu\|x\| \\
{\rm s.t.} \quad & h(Ax - b) \leq \sigma,\\
           \quad & \|x\|_{\infty} \leq M,
\end{aligned}
\end{equation}
where $M: = (1 - \mu)^{-1}\Big(\|A^\dagger b\|_1 - \mu\|A^\dagger b\|\Big)$.\footnote{We also recall from \cite[Section~6.1]{zhang23} that $\|A^\dagger b\|_\infty\le M$ by the construction of $M$.} Notice that the equivalent problem \eqref{CompSen1} is a special case of \eqref{eq1} with $P_1(x) = \|x\|_1$, $P_2(x) = \mu\|x\|$, $m = 1$, $g_1(x) = h(Ax - b)-\sigma$ and $C = \{x: \|x\|_{\infty}\leq M\}$; since $A$ has full row rank and $h(0) = 0< \sigma$, we see that $A^\dagger b \in C\cap \{x: g_1(x) < 0 \}\not=\emptyset$.

Next, we will focus on \eqref{CompSen1} and consider two specific choices of $h$. All numerical experiments are performed in MATLAB R2022a on a 64-bit PC with an Intel(R) Core(TM) i7-10710U CPU (@1.10GHz, 1.61GHz) and 16GB of RAM.

\subsection{$h(\cdot) = \frac{1}{2}\|\cdot\|^2$}
In this subsection, we take $h(\cdot) = \frac{1}{2}\|\cdot\|^2$, then \eqref{CompSen1} becomes
\begin{equation}\label{CompSen1-1}
\begin{aligned}
\min_{x\in\R^n}\quad &\|x\|_1 - \mu\|x\| \\
{\rm s.t.} \quad & 0.5\cdot\|Ax - b\|^2 \leq \sigma,\\
           \quad & \|x\|_{\infty} \leq M.
\end{aligned}
\end{equation}
Notice that $h$ is convex, the Slater condition holds for the feasible region of \eqref{CompSen1-1}, and the origin is not feasible as $\sigma\in(0, \frac{1}{2}\|b\|^2)$. These together with Remark~\ref{rem43} imply that Assumptions~\ref{A1} and \ref{A2} hold. Since the $H$ in \eqref{defH} corresponding to \eqref{CompSen1-1} is clearly semi-algebraic and hence a KL function, one can then apply Theorem~\ref{th2.2} with $\ell_g = 0$ to deduce the convergence of the (whole) sequence $\{x^k\}$ generated by Algorithm~\ref{alg:Framwork} with $\sup_k\beta_k < 1$ for solving \eqref{CompSen1-1}.\footnote{Though we are not considering $\mu = 0$ in our experiments below, we also point out that when $\mu = 0$ in \eqref{CompSen1-1}, thanks to $\sigma\in(0, \frac{1}{2}\|b\|^2)$ and the fact that $A^\dagger b \in C\cap \{x: g_1(x) < 0 \}$, we can deduce from Example~\ref{RemarkKL} that $x\rightarrow \|x\|_1 + \delta_C(x) + \eta [g_1(x)]_+$ is a KL function with exponent $\frac{1}{2}$ whenever $\eta$ exceeds some exact penalty parameter. Therefore, according to Remark~\ref{rem52}, for the sequence $\{(x^k,\theta_k)\}$ generated by Algorithm~\ref{alg:Framwork}, if $\hat\theta$ exceeds some exact penalty parameter, then $\{x^k\}$ converges locally linearly thanks to Theorem~\ref{Edecres}(iii).}

We compare SCP$_{\rm ls}$ in \cite{yu21}, ESQM$_{\rm e}$ (Algorithm~\ref{alg:Framwork} with $\{\beta_k\}$ specified below) and ESQM$_{\rm b}$ (this is a basic version of ESQM obtained by setting $\beta_k \equiv 0$ in Algorithm~\ref{alg:Framwork}). We use the same parameter settings for SCP$_{\rm ls}$ in \cite{yu21}, and the initial point of SCP$_{\rm ls}$ is chosen as $x^0 = A^\dagger b$. For ESQM$_{\rm b}$ and ESQM$_{\rm e}$, we take $L_g = \|A\|^2$, $\ell_g = 0$, $d = 1$ and $\theta_0 = 1$, and their initial points are chosen as $x^0 = 0$. We terminate all algorithms when
\begin{equation}\label{termination}
  \|x^{k+1} - x^k\|< \epsilon\cdot\max\{1, \|x^{k+1}\|\}
\end{equation}
for some $\epsilon > 0$ specified below. The subproblems in these algorithms are solved according to the procedures described in the appendices of \cite{yu21} and \cite{zhang23}.

We use the same strategy of choosing $\beta_k$ as in the FISTA with fixed and adaptive restart described in \cite{DonoghueCandes15}. In more detail, we set the initial values $\vartheta_{-1}=\vartheta_{0}=1$ and define, for $k\geq 0$,
\begin{align}\label{beta}
\beta_k = \frac{\vartheta_{k-1}-1}{\vartheta_{k}} \ \ {\rm with}\ \ \vartheta_{k+1}=\frac{1+\sqrt{1+4\vartheta_{k}^2}}{2},
\end{align}
and we reset $\vartheta_{k-1}=\vartheta_{k}=1$ every $K = 200$ iterations or when $\langle y^{k-1} - x^k, x^k - x^{k-1} \rangle>0$. One can show that $\{\beta_k\}$ generated this way satisfies $\{\beta_k\}\subseteq[0,1)$ and $\sup_{k}\beta_k<1$.

We perform tests on random instances of \eqref{CompSen1-1}. Specifically, we generate an $A\in\R^{q\times n}$ with independent and identically distributed (i.i.d.) standard Gaussian entries, and then normalize this matrix so that each column of it has unit norm. Then we choose a subset $T$ of size $k$ uniformly at random from $\{1, 2, \cdots, n\}$ and a $k$-sparse vector $x_{\rm orig}$ having i.i.d. standard Gaussian entries on $T$ is generated. We let $b = Ax_{\rm orig} + 0.01\cdot\hat{n}$ with $\hat{n}$ being a random vector having i.i.d. standard Gaussian entries, and $\sigma = 0.5\sigma_1^2$ with $\sigma_1 = 1.1\cdot\|0.01\cdot\hat{n}\|$.

In our numerical tests, we let $\mu =0.95$ in \eqref{CompSen1-1} and $(q,n,k) = (720i,2560i,160i)$ with $i\in \{2, 4, 6, 8, 10\}$. For each $i$, we generate 20 random instances as described above. We present the computational results when $\epsilon$ in \eqref{termination} equals $10^{-4}$ and $10^{-6}$ in Table~\ref{table1} and Table~\ref{table2}, respectively, averaged over the $20$ random instances. Here, we present the time for computing the QR decomposition of $A^T$ (denoted by $t_{\rm QR}$), the time for computing $\|A\|^2$ (denoted by $t_{\|A\|}$),\footnote{The $\|A\|^2$ is computed via the Matlab code {\sf norm(A*A')} when $p \le 2000$, and is computed using {\sf eigs(A*A',1,'LM')} otherwise.} the time for computing $x^0 = A^\dagger b$ given the QR factorization of $A^T$ (denoted by $t_{A^\dagger b}$),\footnote{Note that $A^\dagger b$ is used by SCP$_{\rm ls}$ as the initial point and for computing the $M$ in \eqref{CompSen1} for ESQM$_{\rm b}$ and ESQM$_{\rm e}$, while $\|A\|$ is only used by ESQM$_{\rm b}$ and ESQM$_{\rm e}$.} the CPU times of the algorithms,\footnote{The CPU times do not include $t_{\rm QR}$, $t_{\|A\|}$ and $t_{A^\dagger b}$.} the number of iterations (denoted by Iter), the recovery errors $\text{RecErr} := \frac{\|x^* - \xorig\|}{\max\{1, \|\xorig\|\}}$ and the residuals ${\rm Residual} := \frac{\|Ax^* - b\|^2 - \sigma^2_1 }{\sigma^2_1}$, where $x^*$ is the approximate solution returned by the respective algorithm.

\begin{table}[h]
{\color{black}
\caption{Computational results for problem \eqref{CompSen1-1} with $\epsilon =10^{-4}$.}\label{table1}
\begin{center}
{\footnotesize
\begin{tabular}{|c|c|c|c|c|c|c|}\hline
\phantom{\diagbox{Date}{$i$}} & \multicolumn{1}{|c|}{Method} & \multicolumn{1}{|c|}{$i = 2$} & \multicolumn{1}{c|}{ $i = 4$ }
& \multicolumn{1}{c|}{ $i = 6$ } & \multicolumn{1}{|c|}{ $i = 8$ } & \multicolumn{1}{c|}{ $i = 10$ }
\\\cline{1-7}\multirow{6}*{CPU time (sec)} & \multirow{1}*{$t_{\rm QR}$}
&  0.615 &  3.713 & 13.600 & 32.645 & 66.354             \\\cline{2-2} \multirow{1}*{} & \multirow{1}*{$t_{A^\dagger b}$}
&  0.006 &  0.024 &  0.059 &  0.113 &  0.176             \\\cline{2-2} \multirow{1}*{} & \multirow{1}*{$t_{\|A\|}$}
&  0.532 &  1.438 &  4.754 & 11.120 & 21.776             \\\cline{2-7} \multirow{1}*{} & \multirow{1}*{SCP$_{\rm ls}$}
&  2.332 &  8.159 & 18.235 & 32.125 & 48.558             \\\cline{2-2} \multirow{1}*{}  & \multirow{1}*{ESQM$_{\rm b}$}
&  8.157 & 34.875 & 84.801 & 143.836 & 234.291             \\\cline{2-2} \multirow{1}*{}  & \multirow{1}*{ESQM$_{{\rm e}}$}
&  0.559 &  2.230 &  5.401 &  9.111 & 14.713             \\\cline{1-7} \multirow{3}*{Iter} & \multirow{1}*{SCP$_{\rm ls}$}
&    208 &    213 &    211 &    212 &    212             \\\cline{2-2} \multirow{1}*{}     & \multirow{1}*{ESQM$_{\rm b}$}
&   1729 &   1781 &   1819 &   1768 &   1789             \\\cline{2-2} \multirow{1}*{}     & \multirow{1}*{ESQM$_{{\rm e}}$}
&    108 &    112 &    114 &    112 &    113             \\\cline{1-7} \multirow{3}*{RecErr} & \multirow{1}*{SCP$_{\rm ls}$}
&  0.053 &  0.053 &  0.054 &  0.055 &  0.055       \\\cline{2-2} \multirow{1}*{} & \multirow{1}*{ESQM$_{\rm b}$}
&  0.070 &  0.071 &  0.073 &  0.073 &  0.074       \\\cline{2-2} \multirow{1}*{} & \multirow{1}*{ESQM$_{{\rm e}}$}
&  0.051 &  0.051 &  0.052 &  0.053 &  0.053       \\\cline{1-7} \multirow{3}*{Residual} & \multirow{1}*{SCP$_{\rm ls}$}
& -1.61e-05 & -2.01e-05 & -2.07e-05 & -2.06e-05 & -2.03e-05      \\\cline{2-2} \multirow{1}*{} & \multirow{1}*{ESQM$_{\rm b}$}
& 6.36e-07 & 6.04e-07 & 5.42e-07 & 5.60e-07 & 5.38e-07     \\\cline{2-2} \multirow{1}*{} & \multirow{1}*{ESQM$_{{\rm e}}$}
& 1.20e-07 & 1.11e-07 & 9.96e-08 & 9.71e-08 & 1.03e-07    \\\cline{1-7}
\end{tabular}
}
\end{center}
}
\end{table}

\begin{table}[h]
{\color{black}
\caption{Computational results for problem \eqref{CompSen1-1} with $\epsilon =10^{-6}$.}\label{table2}
\begin{center}
{\footnotesize
\begin{tabular}{|c|c|c|c|c|c|c|}\hline
\phantom{\diagbox{Date}{$i$}} & \multicolumn{1}{|c|}{Method} & \multicolumn{1}{|c|}{$i = 2$} & \multicolumn{1}{c|}{ $i = 4$ }
& \multicolumn{1}{c|}{ $i = 6$ } & \multicolumn{1}{|c|}{ $i = 8$ } & \multicolumn{1}{c|}{ $i = 10$ }
\\\cline{1-7}\multirow{6}*{CPU time (sec)} & \multirow{1}*{$t_{\rm QR}$}
&  0.662 &  4.444 & 13.801 & 31.694 & 59.477              \\\cline{2-2} \multirow{1}*{} & \multirow{1}*{$t_{A^\dagger b}$}
&  0.007 &  0.029 &  0.060 &  0.104 &  0.160              \\\cline{2-2} \multirow{1}*{} & \multirow{1}*{$t_{\|A\|}$}
&  0.576 &  1.633 &  4.858 & 10.567 & 20.454              \\\cline{2-7} \multirow{1}*{} & \multirow{1}*{SCP$_{\rm ls}$}
&  2.919 & 10.359 & 22.412 & 38.074 & 58.432              \\\cline{2-2} \multirow{1}*{}  & \multirow{1}*{ESQM$_{\rm b}$}
& 12.849 & 57.570 & 137.098 & 232.529 & 368.612              \\\cline{2-2} \multirow{1}*{}  & \multirow{1}*{ESQM$_{{\rm e}}$}
&  0.936 &  4.470 & 10.606 & 18.551 & 29.806              \\\cline{1-7} \multirow{3}*{Iter} & \multirow{1}*{SCP$_{\rm ls}$}
&    251 &    257 &    257 &    258 &    259              \\\cline{2-2} \multirow{1}*{}     & \multirow{1}*{ESQM$_{\rm b}$}
&   2756 &   2860 &   2954 &   2887 &   2924              \\\cline{2-2} \multirow{1}*{}     & \multirow{1}*{ESQM$_{{\rm e}}$}
&    195 &    220 &    228 &    230 &    237              \\\cline{1-7} \multirow{3}*{RecErr} & \multirow{1}*{SCP$_{\rm ls}$}
&  0.051 &  0.051 &  0.052 &  0.053 &  0.053           \\\cline{2-2} \multirow{1}*{} & \multirow{1}*{ESQM$_{\rm b}$}
&  0.051 &  0.051 &  0.052 &  0.053 &  0.053           \\\cline{2-2} \multirow{1}*{} & \multirow{1}*{ESQM$_{{\rm e}}$}
&  0.051 &  0.051 &  0.052 &  0.052 &  0.053           \\\cline{1-7} \multirow{3}*{Residual} & \multirow{1}*{SCP$_{\rm ls}$}
& -1.61e-09 & -1.86e-09 & -1.85e-09 & -1.67e-09 & -1.84e-09      \\\cline{2-2} \multirow{1}*{} & \multirow{1}*{ESQM$_{\rm b}$}
& 9.09e-11 & 9.04e-11 & 8.71e-11 & 8.74e-11 & 8.85e-11   \\\cline{2-2} \multirow{1}*{} & \multirow{1}*{ESQM$_{{\rm e}}$}
& 5.66e-11 & 1.00e-10 & -4.51e-14 & 4.10e-11 & -4.93e-13   \\\cline{1-7}
\end{tabular}
}
\end{center}
}
\end{table}

From Table~\ref{table1} and Table~\ref{table2}, one can see that ESQM$_{{\rm e}}$ is the fastest algorithm, and the recovery errors of all three methods are comparable.

\subsection{When $h$ is the Lorentzian norm}

In this subsection, we consider $h$ being the Lorentzian norm \cite{CarrBarAys10}, which is defined as follows for any given $\gamma>0$:
\[
\|y\|_{L L_2, \gamma}:=\sum_{i=1}^q \log\left(1 + \frac{y_i^2}{\gamma^2}\right).
\]
Then, problem~\eqref{CompSen1} becomes the following problem:
\begin{equation}\label{CompSen1-2}
\begin{aligned}
\min_{x\in\R^n}\quad & \|x\|_1 - \mu\|x\| \\
{\rm s.t.}\quad & \|Ax - b\|_{L L_2, \gamma}\leq \sigma,\\
\quad & \|x\|_{\infty} \le M.
\end{aligned}
\end{equation}
We first argue that Assumption~\ref{A1} holds for \eqref{CompSen1-2} under our assumptions on $A$ and $\sigma$ in \eqref{CompSen1}. To this end, let $\hat h(y) := \|y\|_{L L_2, \gamma} - \sigma$ for notational simplicity. Then \eqref{CompSen1-2} is an instance of \eqref{eq1} with $g_1(x) = \hat h(Ax - b) - \sigma$ and $C := \{x:\; \|x\|_\infty \le M\}$. Next, recall that $A^\dagger b\in C$ by the construction of $M$. Moreover, observe that for any $x\in C$, we have
\begin{align}\label{RCQexpression}
&\langle \nabla g_1(x), A^\dagger b - x\rangle = \langle A^T\nabla \hat h(Ax-b), A^\dagger b - x\rangle =  \langle \nabla \hat h(Ax-b), b - Ax\rangle\notag\\
& = 2\sum_{i=1}^q\frac{a_i^Tx - b_i}{\gamma^2 + (a_i^Tx - b_i)^2}\cdot (b_i - a_i^Tx)
= -2\sum_{i=1}^q\frac{(a_i^Tx - b_i)^2}{\gamma^2 + (a_i^Tx - b_i)^2},
\end{align}
where $a_i^T$ is the $i$-th row of $A$. Now we consider two cases:
\begin{itemize}
  \item $x\in C\setminus \mathscr{F}$. In this case, suppose that there exist $u_i$, $i\in I(x)$, such that \eqref{A11} holds. Then, in particular, we must have $I(x)$ (defined in Assumption~\ref{A1}) being nonempty, which in turn means $I(x) = \{1\}$. In addition, \eqref{A11} together with \eqref{RCQexpression} implies that $a_i^Tx = b_i$ for all $i$. But then $g_1(x) = \hat h(0) - \sigma = -\sigma < 0$, contradicting the fact that $I(x) = \{1\}$.
  \item $x\in C\cap \mathscr{F}$. In this case, we claim that $g_1(x) + \langle \nabla g_1(x), A^\dagger b - x\rangle < 0$. Suppose to the contrary that $g_1(x) + \langle \nabla g_1(x), A^\dagger b - x\rangle = 0$. Then using $x\in \mathscr{F}$ and \eqref{RCQexpression}, we deduce that $g_1(x) = \langle \nabla g_1(x), A^\dagger b - x\rangle = 0$. The second equality together with \eqref{RCQexpression} implies that $a_i^Tx = b_i$ for all $i$. But then we deduce from this and $g_1(x) = 0$ that
      \[
       0 = g_1(x) = \hat h(0) - \sigma = -\sigma < 0,
      \]
      which is a contradiction. Thus, we have shown that $RCQ(x)$ holds and one can actually choose $y = A^\dagger b$ there.
\end{itemize}
Consequently, Assumption~\ref{A1} holds.

Next, observe that the $\hat h$ has Lipschitz continuous gradient with modulus $\frac{2}{\gamma^2}$. The following proposition shows that $\hat h$ can be represented as the difference of two convex functions $\hat h_1$ and $\hat h_2$ with Lipschitz continuous gradients, and the Lipschitz continuity modulus of $\nabla \hat h_1$ is $\frac{2}{\gamma^2}$ while that of $\nabla \hat h_2$ is $\frac{1}{4\gamma^2}$.
\begin{proposition}
Let $\hat h(y): = \|y\|_{L L_2, \gamma} - \sigma$. Then there exist two convex functions $\hat h_1$ and $\hat h_2$ with Lipschitz continuous gradients such that $\hat h(y) = \hat h_1(y) - \hat h_2(y)$ and the Lipschitz continuity modulus of $\nabla \hat h_1$ is $\frac{2}{\gamma^2}$ while that of $\nabla \hat h_2$ is $\frac{1}{4\gamma^2}$.
\end{proposition}

\begin{proof}
First, notice that
\[
\frac{d^2}{dt^2}\log(1 + t^2) = \frac{2(1 - t^2)}{(1 + t^2)^2} = \left[\frac{2(1 - t^2)}{(1+t^2)^2}\right]_+ - \left[\frac{2(1 - t^2)}{(1+t^2)^2}\right]_-,
\]
where $s_+ := \max\{s,0\}\ge 0$ and $s_- := -\min\{s,0\}\ge 0$ for any $s\in \R$.
Now, define, for each $t\in \R$,
\[
r_1(t) =\int_0^t (t-s)\left[\frac{2(1 - s^2)}{(1+s^2)^2}\right]_+ds{\text{~ and ~}} r_2(t) =\int_0^t (t-s)\left[\frac{2(1 - s^2)}{(1+s^2)^2}\right]_-ds.
\]
Then
$r_1''(t) = \left[\frac{2(1 - t^2)}{(1 + t^2)^2}\right]_+$ and $r_2''(t) = \left[\frac{2(1 - t^2)}{(1 + t^2)^2}\right]_-$, showing that $r_1$ and $r_2$ are convex. Moreover, one can observe that $r_1(0) = r_2(0) = r'_1(0) = r'_2(0) = 0$, and a direct computation shows that $\log(1+t^2) = r_1(t) - r_2(t)$, $\sup_t |r_1''(t)| = 2$ and $\sup_{t}|r_2''(t)| = \frac{1}{4}$.
Taking
\[
\hat h_1(y) = \sum_{i = 1}^m r_1(y_i/\gamma) - \sigma, \text{~  and  ~} \hat h_2(y) = \sum_{i = 1}^m r_2(y_i/\gamma),
\]
one can see that $\hat h_1$ and $\hat h_2$ are two convex functions with Lipschitz continuous gradients, and $\hat h(y) = \hat h_1(y) - \hat h_2(y)$. Furthermore, the Lipschitz continuity modulus of $\nabla \hat h_1$ and $\nabla \hat h_2$ are $\frac{2}{\gamma^2}$ and $\frac{1}{4\gamma^2}$, respectively.
\end{proof}
Recall that the origin is not feasible for \eqref{CompSen1-2} under our assumptions on $A$ and $\sigma$ in \eqref{CompSen1}. In view of this, the above discussions, and the observation that the $H$ in \eqref{defH} corresponding to \eqref{CompSen1-1} is a subanalytic function that is continuous on its closed domain (and hence a KL function in view of \cite[Theorem~3.1]{bolte07}), one can apply Theorem~\ref{th2.2} with $L_g = \frac{2\|A\|^2}{\gamma^2}$ and $\ell_g = \frac{\|A\|^2}{4\gamma^2}$ to deduce the convergence of the $\{x^k\}$ generated by Algorithm~\ref{alg:Framwork} with $\sup_{k}\beta_k < \sqrt{\frac{L_g}{L_g + \ell_g}} = \sqrt{\frac89}$ for solving \eqref{CompSen1-2}.

As in the previous subsection, we compare SCP$_{\rm ls}$, ESQM$_{\rm b}$ and ESQM$_{\rm e}$. For SCP$_{\rm ls}$, we use the same parameter settings in \cite{yu21}, and initialize it at $x^0 = A^\dagger b$. For ESQM$_{\rm b}$ and ESQM$_{\rm e}$, we take $L_g = \frac{2\|A\|^2}{\gamma^2}$, $\ell_g = \frac{\|A\|^2}{4\gamma^2}$, $d = \frac{\gamma^2}{150\|A\|^2}$ and $\theta_0 = 1.1\gamma$, and they are initialized at $x^0 = 0$. We terminate all algorithms when \eqref{termination} holds for some $\epsilon > 0$ specified below.  Furthermore, the subproblems in these algorithms are solved according to the procedures described in the appendices of \cite{yu21} and \cite{zhang23}.

We also choose $\{\beta_k\}$ as described in \eqref{beta} but we set the fixed restart frequency as $K = 48$. This parameter will ensure that $\{\beta_k\}$ satisfies $\{\beta_k\}\subseteq\left[0,\sqrt{\frac{L_g}{L_g + \ell_g}}\right)$ and $ \sup_{k}\beta_k<\sqrt{\frac{L_g}{L_g + \ell_g}}$.

We perform tests on random instances of \eqref{CompSen1-1}. As in the previous section, we generate an $A\in\R^{q\times n}$ with i.i.d. standard Gaussian entries, and then normalize its columns. We then choose a subset $T$ of size $k$ uniformly at random from $\{1, 2, \cdots, n\}$ and generate a $k$-sparse vector $x_{\rm orig}$ with i.i.d. standard Gaussian entries on $T$. We let $b = Ax_{\rm orig} + 0.01\cdot\bar{n}$ with $\bar{n}_i\sim{\rm Cauchy}(0, 1)$, specifically, we generate $\bar{n}_i$ as $\tan(\pi(\tilde{n}_i - \frac{1}{2}))$ with $\tilde{n}$ being a random vector with i.i.d. entries uniformly chosen in $[0, 1]$. We then set $\sigma = 1.05\cdot\|0.01\cdot\bar{n}\|_{L L_2, \gamma}$ with $\gamma = 0.08$.

In our numerical tests, we let $\mu =0.95$ in \eqref{CompSen1-2} and $(q,n,k) = (720i,2560i,80i)$ with $i\in \{2, 4, 6, 8, 10\}$. For each $i$, we generate 20 random instances as described above. The computational results for $\epsilon$ in \eqref{termination} being $10^{-4}$ and $10^{-6}$ are respectively presented in Table~\ref{table3} and Table~\ref{table4}, averaged over the $20$ random instances. As before, we present the time for computing the QR decomposition of $A^T$ (denoted by $t_{\rm QR}$), the time for computing $\|A\|^2$ (denoted by $t_{\|A\|}$), the time for computing $x^0 = A^\dagger b$ given the QR factorization of $A^T$ (denoted by $t_{A^\dagger b}$), the CPU times of the algorithms,\footnote{The CPU times do not include $t_{\rm QR}$, $t_{\|A\|}$ and $t_{A^\dagger b}$.} the number of iterations (denoted by Iter), the recovery errors $\text{RecErr} := \frac{\|x^* - \xorig\|}{\max\{1, \|\xorig\|\}}$ and the residuals ${\rm Residual} := \frac{\|Ax^* - b\|_{LL_2,\gamma} - \sigma }{\sigma}$, where $x^*$ is the approximate solution returned by the respective algorithm.

\begin{table}[h]
{\color{black}
\caption{Computational results for problem \eqref{CompSen1-2} with $\varepsilon = 10^{-4}$. }\label{table3}
\begin{center}
{\footnotesize
\begin{tabular}{|c|c|c|c|c|c|c|}\hline
\phantom{\diagbox{Date}{$i$}} & \multicolumn{1}{|c|}{Method} & \multicolumn{1}{|c|}{$i = 2$} & \multicolumn{1}{c|}{ $i = 4$ }
& \multicolumn{1}{c|}{ $i = 6$ } & \multicolumn{1}{|c|}{ $i = 8$ } & \multicolumn{1}{c|}{ $i = 10$ }
\\ \cline{1-7}\multirow{6}*{CPU time (sec)} & \multirow{1}*{$t_{\rm QR}$}
&  0.577 &  4.143 & 11.775 & 27.407 & 50.278             \\ \cline{2-2} \multirow{1}*{} & \multirow{1}*{$t_{A^\dagger b}$}
&  0.005 &  0.026 &  0.049 &  0.088 &  0.140             \\ \cline{2-2} \multirow{1}*{} & \multirow{1}*{$t_{\|A\|}$}
&  0.467 &  1.548 &  4.432 &  9.593 & 17.722             \\ \cline{2-7} \multirow{1}*{} & \multirow{1}*{SCP$_{\rm ls}$}
&  1.186 &  6.915 &  8.429 & 45.353 & 29.767             \\ \cline{2-2} \multirow{1}*{}  & \multirow{1}*{ESQM$_{\rm b}$}
&  2.587 & 11.985 & 26.931 & 47.601 & 75.835             \\ \cline{2-2} \multirow{1}*{}  & \multirow{1}*{ESQM$_{{\rm e}}$}
&  0.561 &  2.557 &  5.635 &  9.984 & 15.804             \\ \cline{1-7} \multirow{3}*{Iter} & \multirow{1}*{SCP$_{\rm ls}$}
&    120 &    195 &    110 &    354 &    153             \\ \cline{2-2} \multirow{1}*{}     & \multirow{1}*{ESQM$_{\rm b}$}
&    586 &    609 &    607 &    608 &    613             \\ \cline{2-2} \multirow{1}*{}     & \multirow{1}*{ESQM$_{{\rm e}}$}
&    120 &    126 &    125 &    126 &    127             \\ \cline{1-7} \multirow{3}*{RecErr} & \multirow{1}*{SCP$_{\rm ls}$}
&  0.092 &  0.090 &  0.092 &  0.092 &  0.092       \\ \cline{2-2} \multirow{1}*{} & \multirow{1}*{ESQM$_{\rm b}$}
&  0.096 &  0.093 &  0.095 &  0.095 &  0.096       \\ \cline{2-2} \multirow{1}*{} & \multirow{1}*{ESQM$_{{\rm e}}$}
&  0.092 &  0.089 &  0.091 &  0.091 &  0.092       \\ \cline{1-7} \multirow{3}*{Residual} & \multirow{1}*{SCP$_{\rm ls}$}
& -1.91e-07 & -2.26e-07 & -2.31e-07 & -2.68e-07 & -2.74e-07  \\ \cline{2-2} \multirow{1}*{} & \multirow{1}*{ESQM$_{\rm b}$}
& 8.81e-08 & 8.86e-08 & 8.58e-08 & 8.61e-08 & 8.51e-08   \\ \cline{2-2} \multirow{1}*{} & \multirow{1}*{ESQM$_{{\rm e}}$}
& 1.02e-08 & 1.23e-08 & 1.46e-08 & 5.75e-09 & 6.40e-09   \\ \cline{1-7}
\end{tabular}
}
\end{center}
}
\end{table}

\begin{table}[h]
{\color{black}
\caption{Computational results for problem \eqref{CompSen1-2} with $\varepsilon = 10^{-6}$. }\label{table4}
\begin{center}
{\footnotesize
\begin{tabular}{|c|c|c|c|c|c|c|}\hline
\phantom{\diagbox{Date}{$i$}} & \multicolumn{1}{|c|}{Method} & \multicolumn{1}{|c|}{$i = 2$} & \multicolumn{1}{c|}{ $i = 4$ }
& \multicolumn{1}{c|}{ $i = 6$ } & \multicolumn{1}{|c|}{ $i = 8$ } & \multicolumn{1}{c|}{ $i = 10$ }
\\ \cline{1-7}\multirow{6}*{CPU time (sec)} & \multirow{1}*{$t_{\rm QR}$}
&  0.558 &  4.093 & 12.902 & 31.823 & 61.492             \\ \cline{2-2} \multirow{1}*{} & \multirow{1}*{$t_{A^\dagger b}$}
&  0.006 &  0.029 &  0.059 &  0.112 &  0.180             \\ \cline{2-2} \multirow{1}*{} & \multirow{1}*{$t_{\|A\|}$}
&  0.466 &  1.546 &  4.660 & 10.864 & 21.334             \\ \cline{2-7} \multirow{1}*{} & \multirow{1}*{SCP$_{\rm ls}$}
&  1.338 &  8.106 &  9.814 & 48.802 & 36.869             \\ \cline{2-2} \multirow{1}*{}  & \multirow{1}*{ESQM$_{\rm b}$}
&  4.006 & 19.608 & 41.110 & 72.591 & 120.954             \\ \cline{2-2} \multirow{1}*{}  & \multirow{1}*{ESQM$_{{\rm e}}$}
&  0.766 &  3.765 &  7.715 & 13.936 & 23.430             \\ \cline{1-7} \multirow{3}*{Iter} & \multirow{1}*{SCP$_{\rm ls}$}
&    136 &    214 &    127 &    372 &    171             \\ \cline{2-2} \multirow{1}*{}     & \multirow{1}*{ESQM$_{\rm b}$}
&    882 &    914 &    914 &    914 &    919             \\ \cline{2-2} \multirow{1}*{}     & \multirow{1}*{ESQM$_{{\rm e}}$}
&    164 &    169 &    168 &    173 &    174             \\ \cline{1-7} \multirow{3}*{RecErr} & \multirow{1}*{SCP$_{\rm ls}$}
&  0.092 &  0.089 &  0.091 &  0.091 &  0.092       \\ \cline{2-2} \multirow{1}*{} & \multirow{1}*{ESQM$_{\rm b}$}
&  0.092 &  0.089 &  0.091 &  0.091 &  0.092       \\ \cline{2-2} \multirow{1}*{} & \multirow{1}*{ESQM$_{{\rm e}}$}
&  0.092 &  0.089 &  0.091 &  0.091 &  0.092       \\ \cline{1-7} \multirow{3}*{Residual} & \multirow{1}*{SCP$_{\rm ls}$}
& -2.31e-11 & -2.37e-11 & -3.19e-11 & -2.94e-11 & -1.98e-11  \\ \cline{2-2} \multirow{1}*{} & \multirow{1}*{ESQM$_{\rm b}$}
& 8.62e-12 & 8.68e-12 & 8.39e-12 & 8.44e-12 & 8.29e-12    \\ \cline{2-2} \multirow{1}*{} & \multirow{1}*{ESQM$_{{\rm e}}$}
& 2.23e-12 & 3.59e-12 & 3.89e-12 & 7.46e-13 & 1.19e-12    \\ \cline{1-7}
\end{tabular}
}
\end{center}
}
\end{table}

From Table~\ref{table3} and Table~\ref{table4}, we observe a similar pattern as shown in Table~\ref{table1} and Table~\ref{table2}, i.e., ESQM$_{{\rm e}}$ is the fastest algorithm, and the recovery errors of all three methods are comparable.

\bmhead{Acknowledgments}
The first author is supported in part by the National Natural Science Foundation of China (11901414) and (11871359). The second author is supported in part by the Hong Kong Research Grants Council PolyU153001/22p.

%


\begin{thebibliography}{99}
\bibitem{attouch10}
H. Attouch, J. Bolte, P. Redont and A. Soubeyran.
\newblock Proximal alternating minimization and projection methods for nonconvex problems: An approach based on the Kurdyka-{\L}ojasiewicz inequality.
\newblock {\em Mathematics of Operations Research} 35, 438--457, 2010.


\bibitem{attouch13}
H. Attouch, J. Bolte and B. F. Svaiter.
\newblock Convergence of descent methods for semi-algebraic and tame problems: proximal algorithms, forward--backward splitting, and regularized Gauss--Seidel methods.
\newblock {\em Mathematical Programming} 137, 91--129, 2013.


\bibitem{Ausleder13}
A. Auslender.
\newblock An extended sequential quadratically constrained quadratic programming algorithm for nonlinear, semidefinite, and second-order cone programming.
\newblock{\em Journal of Optimization Theory and Applications} 156, 183--212, 2013.

\bibitem{AusSheTeb10}
A. Auslender, B. Shefi and M. Teboulle.
\newblock A moving balls approximation method for a class of smooth constrained minimization problems.
\newblock {\em SIAM Journal on Optimization} 20, 3232--3259, 2010.



\bibitem{BauschkeBorweinLi99}
H. H. Bauschke, J. M. Borwein and W. Li.
\newblock Strong conical hull intersection property, bounded linear regularity, Jameson's property (G), and error bounds in convex optimization.
\newblock {\em Mathematical Programming} 86, 135--160, 1999.


\bibitem{beck09}
A. Beck and M. Teboulle.
\newblock Fast gradient-based algorithms for constrained total variation image denoising and deblurring problems.
\newblock {\em IEEE Transactions on Image Processing} 18, 2419--2434, 2009.

\bibitem{BeckerCandesGrant11}
S. R. Becker, E. J. Cand\`{e}s and M. C. Grant.
\newblock Templates for convex cone problems with applications to sparse signal recovery.
\newblock {\em Mathematical Programming Computation} 3, 165--218, 2011.


\bibitem{bolte07}
J. Bolte, A. Daniilidis and A. Lewis.
\newblock The {\L}ojasiewicz inequality for nonsmooth subanalytic functions with applications to subgradient dynamical systems.
\newblock {\em SIAM Journal on Optimization} 17, 1205--1223, 2007.

\bibitem{bolte07_2}
J. Bolte, A. Daniilidis, A. Lewis and M. Shiota.
\newblock Clarke subgradients of stratifiable functions.
\newblock {\em SIAM Journal on Optimization} 18, 556--572, 2007.

\bibitem{bolte17}
J. Bolte, T. P. Nguyen, J. Peypouquet and B. W. Suter.
\newblock From error bounds to the complexity of first-order descent methods for convex functions.
\newblock {\em Mathematical Programming} 165, 471--507, 2017.

\bibitem{bolte16}
J. Bolte and E. Pauwels.
\newblock Majorization-minimization procedures and convergence of SQP methods for semi-algebraic and tame programs.
\newblock {\em Mathematics of Operations Research} 41, 442--465, 2016.


\bibitem{bolte14}
J. Bolte, S. Sabach and M. Teboulle.
\newblock Proximal alternating linearized minimization for nonconvex and nonsmooth problems.
\newblock {\em Mathematical Programming} 146, 459--494, 2014.

\bibitem{Brezinski00}
C. Brezinski.
\newblock Convergence acceleration during the 20th century.
\newblock {\em Journal of Computational and Applied Mathematics} 122, 1--21, 2000.

\bibitem{Ca18}
E. J. Cand\`{e}s.
\newblock The restricted isometry property and its implications for compressed sensing.
\newblock {\em  Comptes Rendus Mathematique} 346, 589--592, 2008.

\bibitem{CarrBarAys10}
R. E. Carrillo, K. E. Barner and T. C. Aysal.
\newblock Robust sampling and reconstruction methodsfor sparse signals in the presence of impulsive noise.
\newblock {\em IEEE Journal of Selected Topics in Signal Processing} 4, 392--408, 2010.

\bibitem{ChenLuPong16}
X. Chen, Z. Lu and T. K. Pong.
\newblock Penalty methods for a class of non-Lipschitz optimization problems.
\newblock {\em SIAM Journal on Optimization} 26, 1465--1492, 2016.

\bibitem{DonoghueCandes15}
B. O'Donoghue and E. J. Cand\`{e}s.
\newblock Adaptive restart for accelerated gradient schemes.
\newblock {\em Foundations of Computational Mathematics} 15, 715--732, 2015.

\bibitem{GillWong12}
P. E. Gill and E. Wong.
\newblock Sequential quadratic programming methods.
\newblock In: Lee, J., Leyffer, S. (eds.) {\em Mixed
Integer Nonlinear Programming}, 147--224. Springer, New York, 2012.

\bibitem{SOR00}
A. Hadjidimos.
\newblock Successive overrelaxation (SOR) and related methods.
\newblock {\em Journal of Computational and Applied Mathematics} 123, 177--199, 2000.

\bibitem{HoLS13}
J. D. W. Hosmer, S. Lemeshow and R. X. Sturdivant.
\newblock {\em  Applied Logistic Regression}.
\newblock John Wiley Sons, 3rd edition, 2013.

\bibitem{li18}
G. Li and T. K. Pong.
\newblock Calculus of the exponent of Kurdyka--{\L}ojasiewicz inequality and its applications to linear convergence of first-order methods.
\newblock {\em Foundations of Computational Mathematics} 18, 1199--1232, 2018.

\bibitem{LionsMercier79}
P. L. Lions and B. Mercier.
\newblock Splitting algorithms for the sum of two nonlinear operators.
\newblock {\em SIAM Journal on Numerical Analysis} 16, 964--979, 1979.


\bibitem{nesterov83}
Y. Nesterov.
\newblock A method for unconstrained convex minimization problem with the rate of convergence O ($\frac{1}{k^2}$).
\newblock {\em Doklady AN USSR} 269, 543--547, 1983.

\bibitem{Nesterov2004}
\newblock Y. Nesterov.
\newblock {\em Introductory Lectures on Convex Optimization: A Basic Course.}
\newblock Kluwer Academic Publishers, Boston, 2004.

\bibitem{Nesterov2005}
\newblock Y. Nesterov.
\newblock Smooth minimization of non-smooth functions.
\newblock {\em Mathematical Programming} 103, 127--152, 2005.

\bibitem{Nesterov2013}
\newblock Y. Nesterov.
\newblock Gradient methods for minimizing composite objective function.
\newblock {\em Mathematical Programming} 140, 125--161, 2013.

\bibitem{Polyak64}
B. T. Polyak.
\newblock Some methods of speeding up the convergence of iteration methods.
\newblock {\em USSR Computational Mathematics and Mathematical Physics} 4(5), 1--17, 1964.

\bibitem{Rob75}
S. M. Robinson.
\newblock An application of error bounds for convex programming in a linear space.
\newblock {\em SIAM Journal on Control} 13, 271--273, 1975.

\bibitem{Ro70}
R. T. Rockafellar.
\newblock {\em Convex Analysis.}
\newblock Princeton University Press, Princeton, 1970.

\bibitem{rock97a}
R. T. Rockafellar and R. J-B. Wets.
\newblock {\em Variational Analysis.}
\newblock Springer, 1997.

\bibitem{SmithFordSidi87}
D. A. Smith, W. F. Ford  and A. Sidi.
\newblock Extrapolation methods for vector sequences.
\newblock {\em SIAM Review} 29, 199--233, 1987.


\bibitem{wen17}
B. Wen, X. Chen and T. K. Pong.
\newblock Linear convergence of proximal gradient algorithm with extrapolation for a class of nonconvex nonsmooth minimization problems.
\newblock {\em SIAM Journal on Optimization} 27, 124--145, 2017.

\bibitem{wen18}
B. Wen, X. Chen, and T. K. Pong.
\newblock A proximal difference-of-convex algorithm with extrapolation.
\newblock {\em Computational Optimization and Applications} 69, 297--324, 2018.

\bibitem{YinLouHeXin15}
P. Yin, Y. Lou, Q. He and J. Xin.
\newblock Minimization of $\ell_{1-2}$ for compressed sensing.
\newblock {\em SIAM Journal on Scientific Computing} 37, A536--A563, 2015.


\bibitem{yu21}
P. Yu, T. K. Pong and Z. Lu.
\newblock Convergence rate analysis of a sequential convex programming method with line search for a class of constrained difference-of-convex optimization problems.
\newblock {\em SIAM Journal on Optimization} 31, 2024--2054, 2021.

\bibitem{zhang23}
Y. Zhang, G. Li, T. K. Pong and S. Xu.
\newblock Retraction-based first-order feasible methods for difference-of-convex programs with smooth inequality and simple geometric constraints.
\newblock {\em Advances in Computational Mathematics} 49, Article number: 8, 2023.





\end{thebibliography}
\end{document}